\documentclass{amsart}
\usepackage{tikz-cd}
\usepackage{amscd, amsfonts, amsmath, amssymb, amsthm}
\usepackage[noadjust]{cite}
\title{An averaging formula for Nielsen numbers on infra-solvmanifolds}
\author{
Karel Dekimpe}\thanks{Research supported by long term structural funding - Methusalem grant of the Flemish Government.}
\author{Iris van Den Bussche}
\address{
 {KU Leuven Campus Kulak Kortrijk},
{8500 Kortrijk},
{Belgium}}
\email{Karel.Dekimpe@kuleuven.be}
\email{iris.vandenbussche@kuleuven.be}

\newcommand{\p}{\varphi}
\newcommand{\al}{\alpha}
\newcommand{\clp}[1]{[#1]_{\varphi}}
\newcommand{\cl}[2]{[#1]_{#2}}

\newcommand{\R}{\mathbb{R}}
\newcommand{\C}{\mathbb{C}}
\newcommand{\Q}{\mathbb{Q}}
\newcommand{\RR}[1]{\mathcal{R}(#1)}

\newcommand{\Aff}{\mathrm{Aff}}
\newcommand{\aff}{\mathrm{aff}}
\newcommand{\Aut}[1]{\mathrm{Aut}(#1)}

\newcommand{\GLZ}[1]{\mathrm{GL}_{#1}(\Z)}
\newcommand{\SLZ}[1]{\mathrm{SL}_{#1}(\Z)}
\newcommand{\GLQ}[1]{\mathrm{GL}_{#1}(\Q)}

\newcommand{\F}[1]{\operatorname{fix}(#1)}
\newcommand{\ti}{\tilde}
\newcommand{\ii}[1]{{#1}^{-1}}
\newcommand{\Z}{\mathbb{Z}}
\newcommand{\N}{\mathbb{N}}

\newcommand{\PI}{\Pi}
\newcommand{\tf}{torsion free }
\newcommand{\nor}{\triangleleft}
\newcommand{\NR}{$\mathcal{NR}$-}
\newcommand{\f}{\leq_{f}}
\newcommand{\G}{\Gamma}

\newcommand{\infra}{infra-solv}
\renewcommand{\sf}{self-map }

\newcommand{\iso}[2]{\sqrt[^{#1}]{#2}}
\newcommand{\PP}{\mathrm{P}(\R^h)}

\newcommand{\s}{strongly torsion free $S$-}
\newcommand{\di}[1]{\lvert\mathrm{det}(I-#1)\rvert}
\newcommand{\de}[1]{\mathrm{det}(I-#1)}
\newcommand{\X}{\R^h/\PI}
\renewcommand{\S}{\R^h/K}
\newcommand{\eig}[1]{\mathrm{eig}(#1)}
\newcommand{\GL}[1]{\operatorname{GL}_{#1}(\C)}

\newcommand{\vp}{polycyclic-by-finite }
\newcommand{\stepb}[1]{\medskip\noindent\textbf{#1}}
\newcommand{\stepi}[1]{\medskip\noindent\textit{#1}}
\newcommand{\ap}{almost $\Phi$-}
\newcommand{\Sn}{\mathcal{S}_n}
\newcommand{\ad}{\mathrm{ad}}
\newcommand{\Ad}{\mathrm{Ad}}
\newcommand{\Lg}{{\mathfrak g}}
\newcommand{\sgn}[1]{\operatorname{sgn}(#1)}
\newcommand{\fix}[1]{\operatorname{fix}(#1)}
\newcommand{\pd}[1]{\operatorname{P}(\R^{#1})}
\newcommand{\pp}[2]{\operatorname{P}(\R^{#1},\R^{#2})}

\newcommand{\Lm}[1]{\Lambda_{#1}}
\newcommand{\lm}[1]{\lambda_{#1}}
\newcommand{\x}[1]{\overline {x}_{K_{#1}}}
\newcommand{\xx}[2]{{\overline{#1}}_{K_{#2}}}
\newcommand{\RK}[1]{\R^{K_{#1}}}
\newcommand{\Rk}[1]{\R^{k_{#1}}}

\newcommand{\vc}[2]{
\begin{pmatrix}
#1\\
\vdots\\
#2
\end{pmatrix}
}

\newcommand{\vcl}[4]{
\left(
\arraycolsep=0.9pt
\begin{array}{rcl}
#1 &+ &#2\\
 &\vdots &\\
#3 &+ &#4
\end{array}
\right)
}

\newcommand{\End}[1]{\operatorname{End}(#1)}
\newcommand{\eindebewijs}{\hfill\qedsymbol}

\newcommand{\0}[1]{\lvert #1 \rvert_0}
\newcommand{\oo}[1]{\lvert #1 \rvert_\infty}
\newcommand{\cv}{\mathcal{C}(\ti M, \pi, M)}
\newcommand{\K}{K}
\newcommand{\KK}{K'}
\newcommand{\sss}{\hspace*{-0.7mm}}

\newtheorem{theorem}{Theorem}[section]
\newtheorem{prop}[theorem]{Proposition}
\newtheorem{lemma}[theorem]{Lemma}
\newtheorem*{lomma}{Lemma~\ref{matrixanalysis}}

\newtheorem{cor}[theorem]{Corollary}
\theoremstyle{definition}
\newtheorem{definition}[theorem]{Definition}
\newtheorem{example}[theorem]{Example}

\theoremstyle{remark}
\newtheorem{remark}[theorem]{Remark}
\newtheorem{notation}[theorem]{Notation}

\begin{document}
\begin{abstract}
Until now only for special classes of infra-solvmanifolds, namely infra-nilmanifolds and infra-solvmanifolds of type $(R)$, there was a formula available for computing the Nielsen number of a self-map on those manifolds. In this paper, we provide a general averaging formula which works for all self-maps on all possible infra-solvmanifolds and which reduces to the old formulas in the case of infra-nilmanifolds or infra-solvmanifolds of type~$(R)$. Moreover, when viewing an infra-solvmanifold as a polynomial manifold, we recall that any map is homotopic to a polynomial map and we show how our formula can be translated in terms of the Jacobian of that polynomial map.
\end{abstract}
\maketitle

\section{Introduction}

Let $f:M\to M$ be a self-map of a closed manifold $M$. The Lefschetz number of $f$ is defined as the alternating sum of the traces of the induced maps on the homology groups of $M$:
\[ L(f)=\sum_{i=0}^{\dim(M)} (-1)^i {\rm Tr} (f_{\ast,i}:
H_i(M,\R) \to H_i(M,\R)).\]
The famous Lefschetz fixed point theorem (see e.g.\ \cite{hatc02-1}) states that when $L(f)\neq 0$, the map $f$ must have at least one fixed point. 
As $L(f)=L(g)$ whenever $g$ is homotopic to $f$, we have that when $L(f)\neq 0$, any map $g\simeq f$ homotopic to $f$ must have at least one fixed point.

However, the exact value of $L(f)$ does not provide any information on the (least) number of fixed points one should expect for a given map $g \simeq f$.

In Nielsen fixed point theory one tries to overcome this problem by defining a second number, the Nielsen number $N(f)$ of the map $f$, which is also a homotopy invariant and which contains more information than $L(f)$. We refer the reader to \cite{brow71-1,jian83-1,jm06-1} for more information on this number, but vaguely speaking the definition of $N(f)$ goes as follows. 

First, one decomposes the set of fixed points of $f$ into so-called fixed point classes. 
To each of these fixed point classes one then attaches an integer, the index of that fixed point class. A fixed point class is said to be essential if its index is nonzero. The main idea behind this index is that an essential fixed point class can not vanish under a homotopy. 
The Nielsen number $N(f)$ of $f$ is then equal to the number of essential fixed point classes. 

It follows that each map $g\simeq f$ has at least $N(f)$ fixed points. 
If $\dim(M)\geq 3$ we even have more:
\begin{theorem}[Wecken \cite{weck42-1}]
Let $f:M\to M$ be a self-map of a manifold $M$ with $\dim(M)\geq 3$. Then any map $g$ homotopic to $f$ has at least $N(f)$ fixed points. Moreover, there is a map $g$ homotopic to $f$ which has exactly $N(f)$ fixed points.
\end{theorem}

So $N(f)$ contains full information on the least number of fixed points for all maps $g\simeq f$.

In contrast to $L(f)$, the Nielsen number $N(f)$ is unfortunately very hard to compute in general. 
For some classes of manifolds, however, there is a strong relation between the Nielsen number and the Lefschetz number. 
D.~Anosov \cite{anos85-1} showed that $N(f)=|L(f)|$ for any self-map $f:M\to M$ on a nilmanifold $M$. (See also \cite{fh86-1}.) 
Recall that a nilmanifold $M$ is a quotient space $\Gamma \backslash G$, where $G$ is a connected and simply connected nilpotent Lie group and $\Gamma$ is a discrete and uniform subgroup (i.e.~a cocompact lattice) of $G$. 
Anosov's result was generalised to the class of \NR solvmanifolds by E.~Keppelmann and C.~McCord in 1995 \cite{km95-1}. We refer to the next section for more information on this class of manifolds.

An infra-nilmanifold is a manifold that is finitely covered by a nilmanifold. Based on the result of Anosov and using the fact that any self-map of an infra-nilmanifold is homotopic to a so-called affine map, J.B.~Lee and K.B.~Lee \cite{ll09-1} were able to prove a nice formula allowing to compute the Nielsen number of any self-map of an infra-nilmanifold. 
More specifically, they showed that for any infra-nilmanifold $M$ there is a nilmanifold $\tilde{M}$ that finitely covers $M$ and such that any self-map $f:M\to M$ lifts to a self-map $\tilde{f}:\tilde{M}\to \tilde{M}$.
The result of J.B.~Lee and K.B.~Lee then says that $N(f)$ is the average of all $N(\tilde{f})=|L(\tilde{f})|$, where $\tilde{f}$ ranges over all possible lifts of $f$ to $\tilde{M}$. Moreover, one can express $N(f)$ (and $N(\tilde{f})$) easily in terms of the induced morphism $f_\ast: \Pi_1(M) \to \Pi_1(M)$. 

We remark here that there is a slight (and straightforward) generalisation of this result to the class of infra-solvmanifolds of type $(R)$. This is a rather special class of manifolds sharing many properties with the class of infra-nilmanifolds. (See \cite{ll09-1}.)

In this paper, we completely settle the problem of computing the Nielsen number of a self-map on any infra-solvmanifold by proving a general averaging formula for all possible infra-solvmanifolds. 
Moreover, this formula allows us  to express the Nielsen numbers of a self-map $f$ only in terms of the induced morphism $f_\ast: \Pi_1(M) \to \Pi_1(M)$ on the fundamental group.
We also show how to retrieve the previous averaging formulas (for infra-nilmani\-folds and infra-solvmanifolds of type $(R)$) from our general averaging formula.

In addition, we recall that any infra-solvmanifold can be seen as a so-called polynomial manifold and that any self map of such a polynomial manifold is homotopic to a polynomial map. We then translate our general averaging formula to a formula using the Jacobian of the polynomial map.

\section{Infra-solvmanifolds}\label{is}

The manifolds we are studying in this paper are infra-solvmanifolds. These are smooth manifolds which are finitely covered by a solvmanifold. Recall that a solvmanifold is obtained as a quotient of a connected solvable Lie group by a closed subgroup. In this paper, all (infra-)solvmanifolds are assumed to be compact manifolds.

In the literature one can find several possibilities for the exact definition of an infra-solvmanifold; 
Kuroki and Yu present five equivalent definitions in \cite{ky13-1}.

For our purposes, it suffices to know that an infra-solvmanifold is finitely covered by a solvmanifold and that two infra-solvmanifolds with isomorphic fundamental groups are diffeomorphic (see e.g.\ \cite[Corollary 1.5]{baue04-1}). 

\begin{remark}\label{paper of lee}
At this point, we want to remark that some authors use a more restrictive definition of the notion of a solvmanifold and they only use the term solvmanifold for quotients of a simply connected solvable Lie group by a lattice (and not a general closed subgroup). This is a much more restrictive notion and e.g.\ the Klein bottle is not a solvmanifold in this more restricted sense, but it is one for the more general definition used in this paper. For example in \cite{leejb}, the author is studying a class of manifolds which are refered to as infra-solvmanifolds modeled on Sol$_0^4$ and in that paper it is said that these manifolds are not finitely covered by a solvmanifold. However, for the more general definition, these manifolds are in fact themselves solvmanifolds. 
\end{remark}

\subsection{Fundamental group structure}

The algebraic structure of the fundamental groups of solvmanifolds and infra-solvmanifolds is well known. Indeed,
by a result of Wang \cite{wang56-1}, a group $K$ is the fundamental group of a solvmanifold if and only if $K$ fits in a short exact sequence 
\begin{equation}\label{S-group}
1 \to N \to K \to \Z^k\to 1 
\end{equation}
where $N$ is a finitely generated torsion free nilpotent group. 
If $k=0$, then $K=N$ is the fundamental group of a nilmanifold, which is a special type of solvmanifold.
We will refer to a group $K$ fitting in a short exact sequence  of the above form (\ref{S-group}) as a \emph{\s group.} Note that any \s group is a poly--$\Z$ group. 

A group $\PI$ is isomorphic to the fundamental group of 
an infra-solvmanifold if and only if $\PI$ is a torsion free polycyclic-by-finite group (or a torsion free virtually poly--$\Z$ group). 
Every polycyclic-by-finite group $\PI$ admits a series of characteric subgroups  $1\nor  \PI_s\nor\cdots\nor  \PI$ having finite or abelian factors $\PI_i/\PI_{i+1}$. This follows by inductively applying the fact that any infinite polycyclic-by-finite group contains a nontrivial free abelian group as a characteristic subgroup.

\subsection{Infra-solvmanifolds as polynomial manifolds.}\label{polymani}

Let $M$ be an infra-solv\-man\-i\-fold with fundamental group $\PI$. In this paper, we will consider $M=\tilde{M}/\PI$ as being the quotient of the universal covering space $\tilde{M}$ of $M$ by the action of $\PI$ as covering transformations.  

It is known that $\tilde{M}$ is diffeomorphic to $\R^h$ for some $h$. 
In fact, by the work of O.~Baues \cite[Corollary 4.5]{baue04-1}, we can assume that $M = \R^h/\PI$ where the covering group $ \PI$ is acting on $\R^h$ via 
a bounded group of polynomial diffeomorphisms of $\R^h$.
Recall that a map $p:\R^h \to \R^h$ is a \emph{polynomial diffeomorphism} of $\R^h$ if $p$ is bijective and both $p$ and $p^{-1}$ are expressed by means of polynomials in the usual coordinates of $\R^h$. 
Denote by $\pd h$ the group of polynomial diffeomorphisms of $\R^h$. 
 A subgroup $C$ of $\pd h$ is called \emph{bounded} if the degrees of $p$, $p\in C$, have a common bound. An action $\rho:\PI\to\pd h$ is of \emph{bounded degree} if $\rho(\PI)$ is bounded.
We will work with bounded degree actions $\rho:\PI\to\pd h$ in Sections \ref{pol} and \ref{6}.

\section{Nielsen theory}
In this section, we give 
a brief exposition of topological fixed point theory, following the book by Jiang \cite{jian83-1}.

Let $M$ be a compact manifold, and consider a continuous self-map $f:M\to M$ on $M$. 
In Nielsen-Reidemeister theory, one studies the fixed point set $\F f$ %
of $f$ by studying the fixed points of the \emph{lifts} of $f$ to the universal cover $\pi:\ti M\to M$. Indeed, from the commutativity of the diagram
$$\begin{CD}
\tilde M @>\ti f>> \ti M\\
@VV \pi V @VV \pi V\\
M @>f>> M,\\
\end{CD}$$
it easily follows that
$$\F f=\bigcup_{\ti f} \pi(\F {\ti f}),$$
where $\ti f$ ranges over all lifts of $f$.

The lifts of the identity form the \emph{covering transformations} of $\pi:\tilde M\to M$.  These covering transformations form a group, which we denote by $\cv$.
They act on the lifts of $f$ by conjugation.
The corresponding conjugacy classes are called the \emph{lifting classes} of $f$. 
We denote the lifting class of a lift $\ti f$ by $[\ti f]$. 

The interest in lifting classes lies in the following observation:

\medskip
\noindent
{Let $\ti f$, $\ti f'$ be lifts of $f$. }
\begin{enumerate}
    \item {If $\ti f\sim\ti f'$, then $\pi(\F {\ti f})=\pi(\F {\ti f'})$.}

\item {If $\ti f\not \sim\ti f'$, then $\pi(\F {\ti f})\cap \pi(\F {\ti f'})=\emptyset$.}
\end{enumerate}

Accordingly, $\pi(\F {\ti f})$ is called the \emph{fixed point class} of $f$ determined by $[\ti f]$. 
As $\F f=\bigcup_{\ti f}\,\pi(\F {\ti f})$, the fixed point set of $f$ splits into a disjoint union of fixed point classes.

The \emph{Reidemeister number of $f$}, denoted $R(f)$, is the number of lifting classes of $f$, or equivalently,
the number of fixed point classes of $f$. It is either a positive integer or infinite. The Reidemeister number is a homotopy invariant.

The \emph{Nielsen number of }$f$, also a homotopy invariant, is the number of \emph{essential} fixed point classes of $f$. Heuristically an essential fixed point class can be understood as a nonempty fixed point class that never vanishes under a homotopy. The precise definition of an essential fixed point class can be found in \cite{brow71-1, jm06-1, jian83-1}. 

The Reidemeister and Nielsen number of $f$ relate to the number of fixed points of $f$ in the following way:
$$N(f)\leq \#\F f\leq R(f).$$

Lifting classes have the following algebraic characterisation. 
Fix a lift $\ti f_0$ of $f$. 

It is well known that $\cv$ is isomorphic to $\PI:=\Pi_1(M)$, the fundamental group of $M$. If we view elements $\alpha$ of $\Pi_1(M)$ as being covering transformations,
any lift $\tilde f$ of $f$ can be  written uniquely as $\tilde f = \alpha\ti f_0$ for some $\alpha\in\PI$.

In particular, $\tilde{f}_0 \alpha$ is a lift of $f$ for every $\alpha \in \Pi$, so there exists a unique $f_\ast(\alpha) \in \Pi$ satisfying $f_\ast(\alpha) \tilde{f}_0 = \tilde{f}_0\, \alpha$. This defines a morphism $f_\ast : \Pi \to \Pi$\label{induced-endo}, which we call the morphism induced by $f$ on $\Pi$ (with respect to $\tilde f_0$).
We remark in passing that $f_\ast$ agrees with the usual induced morphism $f_\#:\Pi_1(M,x) \to \Pi_1(M,f(x))$, provided one chooses the correct identification of $\Pi_1(M,x)$ and $\Pi_1(M,f(x))$ with $\cv$.

Subsequently, it is easily verified that two lifts $\alpha\ti f_0$, $\beta\ti f_0$ belong to the same lifting class if and only if 
$$\alpha=\gamma\beta f_\ast(\ii\gamma) \text{ for some }\gamma\in\PI,$$
that is, $\alpha$ and $\beta$ are \emph{$f_\ast$-twisted conjugates}.
The above condition defines an equivalence relation $\sim_{f_\ast}$ on $\PI$; let $\RR{f_\ast}:=\PI/\sim_{f_\ast}$ denote the resulting orbit space and let $[\alpha]_{f_\ast}$ denote the equivalence class of $\alpha\in\PI$. 
We have thus established a one-to-one correspondence 
$$\begin{array}{ccc}
    \{\text{lifting classes of }f\} 
    & \overset{1:1}{\longleftrightarrow}
    & \RR{f_\ast}
    \\[0.2cm]
    [\alpha\ti f_0]
    & \longleftrightarrow 
    & [\alpha]_{f_\ast}
\end{array}$$
In particular, the number of ${f_\ast}$-twisted conjugacy classes is given by the Reidemeister number $R(f)$. We also write $R({f_\ast}):=R(f)$ as this number can be computed directly from ${f_\ast}$, and we can define it for any morphism $\p:\PI\to\PI$ via $R(\p):=\#\RR\p$.

Using the above correspondence, we conclude that
$$\F f=\bigsqcup_{[\alpha]_{f_\ast}\in\RR{f_\ast}} \pi(\F {\alpha\ti f_0}).$$

The morphism $f_\ast$ does depend on the chosen lift $\tilde f_0$: if we choose $\alpha\tilde f_0$ as reference lift instead, $f$ induces the morphism $\tau_\alpha\circ f_\ast$
with $\tau_\alpha$ the inner morphism $\tau_\alpha:\PI\to\Pi:\gamma\mapsto \alpha\gamma\ii \alpha$.

In this paper, we will express $N(f)$ solely in terms of the morphism $f_\ast$.
\section{An averaging formula on infra-solvmanifolds}

Let $M$ be any infra-solvmanifold. 
The aim of this section is to show that the Nielsen number of a \sf $f:M\to M$ equals the average of the Nielsen numbers of its lifts to any finite cover $S$ of $M$ satisfying the following two conditions:
\begin{itemize}
    \item every map $f:M\to M$ lifts to a map on  $S$;
    \item $S$ is an \emph{\NR solvmanifold}.
\end{itemize}
\NR solvmanifolds were introduced by Keppelmann and McCord in 1995 \cite{km95-1} as a class of solvmanifolds satisfying the Anosov relation, that is, $N(f)=|L(f)|$ for every \sf $f$. 
We recall the relevant properties of these manifolds in the following subsection. 
We show that $M$ always has a finite cover $S$ satisfying the above two conditions in Subsection~\ref{netex}.

\subsection{\NR solvmanifolds.}

Let $K$ be a \s group. 
Let $N$ be the subgroup $\iso{K}{[K,K]}$, where for a group $G$ and a subgroup $H$ of $G$ we let $\iso{G}{H}$ denote the \textit{isolator} $\{g\in G\mid \exists k\in \N\setminus\{0\}: g^k\in H\}$ of $H$ in $G$.
As $K$ is a \s group, $N$ is nilpotent, say of class $c$.
Let $\gamma_i(N)$ denote the $i$-th term of the lower central series of $N$, and put $N_i:=\iso{N}{\gamma_i(N)}$. 
Then $1\nor N_c\nor\cdots\nor N_1=N$ forms a central series of $N$ with free abelian factors $N_i/N_{i+1}$.
As the $N_i$ are normal subgroups of $K$, we get well-defined actions 
$$\rho_i:K/N\to \Aut{N_i/N_{i+1}}: \bar k\mapsto \, (xN_{i+1}\mapsto kx\ii k N_{i+1})$$ by conjugation. \label{rho_i}

\begin{definition}
We say that $K$ satisfies the \emph{\NR property} if for every $i\in \{1,\ldots,c\}$ and for all $\bar k$ in $K/N$ the automorphism $\rho_i(\bar k)$ (on the free abelian group $N_i/N_{i+1}$) has no nontrivial roots of unity as eigenvalues.
\end{definition}

To see that the \NR property does not really depend on the chosen series $1\nor N_c\nor \cdots\nor N_1=N$ of normal subgroups of $N$, we  introduce the following notations. 
Let $\p:\PI\to \PI$ be an endomorphism on a polycyclic-by-finite group $\PI$. 
Suppose that $$\PI_*: 1=\PI_{s+1}\nor \PI_s\nor \cdots \nor \PI_1=\PI$$ is a normal series of $\PI$ with finite or abelian factors $G_i:=\PI_i/\PI_{i+1}$ such that $\p(\PI_i)\subseteq \PI_i$ for every $i$ in $\{1,\ldots,s\}$. 
Then $\p$ induces endomorphisms $\p_i:G_i\to G_i$, which in turn induce endomorphisms $\bar \p_i$ on $G_i/\tau(G_i)$, where $\tau(G_i)$ is the set of torsion elements of $G_i$.
Note that $\tau(G_i)$ is indeed a subgroup of $G_i$ as $G_i$ is finite or abelian.
The groups $G_i/\tau(G_i)$ are free abelian groups of finite rank. 
Let $\eig {\bar\p_i}$ denote the set of eigenvalues of $\bar\p_i$, where we agree that $\eig{\bar\p_i}=\emptyset$ if $G_i/\tau(G_i)$ is trivial. 

\begin{lemma}
The set $\bigcup_{i=1}^s\eig{\bar\p_i}$ is independent of the chosen series.
\end{lemma}

This lemma can be proved by first showing that the set $\bigcup_{i=1}^s\eig{\bar\p_i}$ does not change if one refines the normal series and then by showing that two different normal series have ``equivalent'' refinements (See \cite[Theorem 8.4.3]{hall76-1}). 

\medskip

Accordingly, we will write $\eig \p:=\bigcup_{i=1}^s\eig{\bar\p_i}$. 
Using this notation, $K$ satisfies the $\mathcal{NR}$-property if and only is $\eig \p$ does not contain a nontrivial root of unity for every inner automorphism $\p$ of $K.$

\begin{definition}
A compact solvmanifold is an \emph{\NR solvmanifold} if its fundamental group satisfies the \NR property.
\end{definition}

Let $f:S\to S$ be a map on an \NR solvmanifold $S$ with fundamental group $K$.  
Suppose that $f$ induces an endomorphism $f_*$ on $K$. Since $N=\iso{K}{[K,K]}$ is a fully characteristic subgroup of $K$, this endomorphism in turn induces an endomorphism $F_0$ on $K/N$ and endomorphisms $F_i$, $i=1,\ldots,c$, on the factor groups $N_i/N_{i+1}.$ 
The collection $\{F_0,\ldots, F_c\}$ is called the \emph{linearisation} of $f_*$.
Keppelmann and McCord  proved the following product formula for Nielsen numbers on \NR solvmanifolds.

\begin{theorem}[Keppelmann-McCord {\cite[Theorem~3.1]{km95-1}}]\label{prod} 
Let $f:S\to S$ be a map on an \NR solvmanifold $S$ with fundamental group $K$. 
Suppose that $f$ induces an endomorphism $f_*$ on $K$ with linearisation  $\{F_0,\ldots, F_c\}$. 
Then 
$$N(f)=\prod_{i=0}^c \,\di{F_i}.$$
\end{theorem}

\begin{remark}\label{nr}
As mentioned above,
the induced endomorphism $f_*$ is not unique: $f$ also induces $\tau_k\circ f_*$ for every inner automorphism $\tau_k:K\to K:x\mapsto kx\ii k$. Let $\rho_0:K/N\to \Aut{K/N}$ denote the trivial map. Then $\tau_k\circ f_*$ has linearisation $\{\rho_0(\bar k)F_0,\dots, \rho_c(\bar k)F_c\}$.
In particular, $\prod_{i=0}^c \,\di{\rho_i(\bar k)F_i}$ is independent of $k$.
\end{remark}

\begin{remark} In Remark~\ref{paper of lee} we already mentioned that the  manifolds in \cite{leejb} are solvmanifolds (for the general definition) and they are in fact 
\NR solvmanifolds. This implies that the Nielsen numbers that were computed in \cite{leejb} using fibering techniques could also be obtained by applying the product formula of Keppelmann and McCord from Theorem~\ref{prod} above.
\end{remark}

\subsection{Finding a suitable cover}\label{netex}

Let $M=\tilde M/\PI$ be an infra-solvmanifold.
In this subsection, we show
that we can always find a fully characteristic, finite index subgroup $K$ of $\PI$ that satisfies the $\mathcal{NR}$-property. 
The resulting quotient space $S:=\tilde{M}/K$ will then 
be an \NR solvmanifold such that every map on $M$ lifts to $S$.

It is well known that $\PI$ admits a fully invariant \s group $\G$ of finite index. 
As before, set $N:=\iso{\G}{[\G, \G]}$, $N_i:=\iso{N}{\gamma_i(N)}$ and consider the actions $\rho_i:\G/N\to \Aut{N_i/N_{i+1}}$ by conjugation. 

The following lemma was proved by Wilking \cite[Lemma~7.5]{wilk00-1} in a more general form.
For the convenience of the reader we adopt his argument to our needs.

\begin{lemma}[Wilking]\label{n}
There exists $n\in \N\setminus \{0\}$ such that for every $\bar\gamma\in\Gamma/N$, the subgroup of $\C^*$ generated by  $\bigcup_{i=1}^{c}\eig{\rho_i(\bar \gamma^n)}$ does not contain a nontrivial root of unity. 
\end{lemma}
\begin{proof}
Take a set of generators $\{z_1, \ldots, z_k\}$ of $\G/N\cong\Z^k$, and consider the set 
$$V:=\{\alpha\in\C\mid \exists\, (i,j)\in\{1,\ldots, c\}\times\{1,\ldots k\}:\alpha\text{ is an eigenvalue of }\rho_i(z_j)\}$$ 
of all the eigenvalues of the $\rho_i(z_j)'s$. 
As $V$ is a finite set of algebraic integers,  the field extension $\Q\subseteq\Q(V)$ is finite.
Hence $\Q(V)$ contains only finitely many roots of unity. Let $n$ denote the number of roots of unity in $\Q(V)$. 
We will show that $n$ satisfies the condition of the lemma. 

Thereto, take $\bar\gamma$ in $\G/N$, 
and let $\mu\in \langle\,\bigcup_{i=1}^{c}\eig{\rho_i(\bar \gamma^n)}\,\rangle_{\C^*}$ be a root of unity, say $\mu^r=1$ for some $r\neq 0$ in $\N$. We have to show that $\mu=1$.
As $\mu\in \langle\,\bigcup_{i=1}^{c}\eig{\rho_i(\bar \gamma^n)}\,\rangle_{\C^*}$, we can write $\mu$ as $\mu=\mu_1^{n_1}\ldots\mu_s^{n_s}$ with $n_j\in\Z$ and $\mu_j$ an eigenvalue of $\rho_{e_j}(\bar \gamma^n)$ for some $e_j\in\{1,\dots, c\}$.
Write each $\mu_j$ as $\mu_j=\lambda_j^n$ with $\lambda_j$ an eigenvalue of $\rho_{e_j}(\bar \gamma)$. 
Then
$$\mu=\left(\lambda_1^{n_1}\ldots\lambda_s^{n_s}\right)^n.$$
Note that the group $H:=\{z^n\mid z\in\Q(V)^*\}$ is torsion free, since the roots of unity in $\Q(V)$ form a subgroup of $\Q(V)^*$ of order~$n$. 
Therefore, the lemma will follow once we show that $\mu\in H$. This follows immediately from the fact that each $\lambda_j\in \Q(V)$.
Indeed, as $z_1, \ldots, z_k$ generate $\G/N$, we can write
$\bar{\gamma}= z_1^{m_1}z_2^{m_2}\ldots z_k^{m_k}$ for some $m_1,m_2,\ldots, m_k\in \Z$ and so 
$\rho_{e_j}(\bar\gamma)=\rho_{e_j}(z_1)^{m_1}\ldots\rho_{e_j}(z_k)^{m_k}$. 
Since $\rho_{e_j}(z_1), \ldots, \rho_{e_j}(z_k)$ commute, $\lambda_j=\alpha_1^{m_1}\ldots\alpha_k^{m_k}$ with $\alpha_l$ some eigenvalue of $\rho_{e_j}(z_l)$. Hence $\lambda_j\in\Q(V)$.

\end{proof}

\begin{theorem}
Let $M$ be an infra-solvmanifold. There exists an \NR solv\-ma\-ni\-fold $S$ that finitely covers $M$ and such that any \sf $f$ of $M$ lifts to a \sf $\tilde{f}$ of $S$.
\end{theorem}

\begin{proof}
Let $M=\tilde{M}/\PI$ be as before, take
$n$ as in Lemma~\ref{n} and let $p:\G\to\G/N$ be the quotient map. 
Consider the subgroup $K:=p^{-1}(\{\bar\gamma^n\mid\bar\gamma\in\G/N\})$ of $\Gamma$.
It is easy to see that $K$ is a fully invariant, finite index subgroup of $\PI$ that satisfies the \NR property. 
As $K$ is fully invariant, any \sf $f$ of $M$ lifts to $S$. 
Therefore $S=\tilde{M}/K$ is the desired \NR solvmanifold. 
\end{proof}

\begin{remark}
By construction, the fundamental group $K$ of $S$ satisfies an even stronger notion called \emph{netness}, see Section~\ref{ni}.  
\end{remark}

\begin{remark}\label{cov-transfo}
Let $M=\tilde{M}/\PI$ and $S=\tilde{M}/K$ be as in the proof above, then 
$S$ is a finite regular cover of $M$ with group of covering transformations $\PI/K$.
\end{remark}

\subsection{Averaging formula over an invariant subgroup.}\label{i}

In \cite{kll05-1}, S.W.\ Kim, J.B.\ Lee and K.B.\ Lee proved the averaging formula for Nielsen numbers on \infra manifolds of type (R). 
We can generalise their averaging formula to any \infra manifold using the main result of \cite{kll05-1}, namely \cite[Theorem~3.1]{kll05-1}.

\begin{theorem}[Averaging formula] \label{avefor}
Let $f$ be a \sf on an infra-solvmanifold $M=\tilde{M}/\Pi$, and let $K$ be a fully invariant, finite index \NR subgroup of $\PI$, so that $S=\tilde{M}/K$ is an \NR solvmanifold covering $M$.
Let $\bar f$ be a lift of $f$ to $S$. Then
$$N(f)=\dfrac{1}{[\PI:K]}\sum_{\bar \alpha\in\PI/K} N(\bar\alpha\bar f).$$
In the formula above, every element $\bar\alpha \in \PI/K$ is acting on $S$ as a covering transformation (see Remark~\ref{cov-transfo}).
\end{theorem}

\begin{proof}
Let $p':\tilde{M}\to S=\tilde{M}/K$ and $p:\tilde{M}\to M= \tilde{M}/\PI$ denote the universal covering projections. 

Let $\tilde{g}:\tilde{M} \to \tilde{M}$ be any lift of $f$ and assume that $\tilde{g}$ induces the endomorphism $\p:\PI\to\PI$ (so 
$\forall \gamma \in \PI: \; \p(\gamma) \tilde{g} = \tilde{g} \gamma$, see page~\pageref{induced-endo}). 
According to \cite[Theorem~3.1]{kll05-1}, we only need to check the following condition:
if $p(\mathrm{Fix}(\tilde g))$ is an essential fixed point class of $f$, then $\mathrm{fix}(\p)\subseteq K$.

 So, let $\tilde g$ be a lifting of $f$,  let $\p$ be the corresponding morphism on $\PI$, and suppose that $p(\mathrm{Fix}(\tilde g))$ is an essential fixed point class of $f$. 
 Let $\bar g$ denote the induced lift on $S$, so we have the commutative diagram
$$ \begin{CD}
  \tilde{M} @>\tilde g>>\tilde{M} \\
  @VV p' V @VV p' V\\
  S @>\bar g>> S \\
    @VV V @VV V\\
  M @>f>> M
  \end{CD}$$
  Note that $\bar g$ induces the endomorphism $\p':=\p|_{K}$ on $K$ with respect to $\tilde g$.  
  We have to show that $\mathrm{fix}(\p)\subseteq K$. 
  We will show in fact that $\mathrm{fix}(\p)$ is trivial. 
  As $K\f \PI$ and $\PI$ is torsion free, it  suffices to show that $\mathrm{fix}(\p')$ is trivial.
  
As $p(\mathrm{Fix}(\tilde g))$ is an essential fixed point class of $f$, also $p'(\mathrm{Fix}(\tilde g))$ is an essential fixed point class of $\bar g$ (see \cite[Remark 2.7]{kll05-1}). 
In particular $N(\bar  g)\neq 0$. Let $\{F_0, \ldots ,F_c\}$ denote the linearisation of $\p'$. As 
$$0\neq N(\bar g)=\prod_{i=0}^c \,\di{F_i}$$ by Theorem~\ref{prod}, $\det(I-F_i)\neq 0$ for all $i\in\{1,\ldots c\}.$ Hence $\mathrm{fix}(F_i)=1$ for all $i$, implying $\mathrm{fix}(\p')=1$ as well.   
\end{proof}

Using Theorem~\ref{prod}, we can make the formula in Theorem~\ref{avefor} more explicit. 
To this end and for further reference, we first introduce the following terminology.

\begin{definition}
Let $\PI$ be a \tf \vp group. A \emph{\tf filtration of $\PI$} is a series of normal subgroups
$$\PI_*: \,1=\PI_{c+1}\nor \PI_c\nor \cdots\nor \PI_0\nor \PI$$
with $\PI/\PI_0$ finite and $\PI_i/\PI_{i+1}$ \tf abelian for all $i\in\{0,\dots, c\}$. See also \cite{dil94-1, di96-1}.
This series induces actions $\lambda_i:\PI\to\Aut{\PI_i/\PI_{i+1}}$ by conjugation. We call the collection $\{\PI_i/\PI_{i+1}, \lambda_i\}_{i=0,\dots, c}$ the \emph{linearisation of $\PI$ induced by the filtration $\PI_*$.}

Let $K\nor_f \PI$ be a \s group. Set $N:=\iso{K}{[K,K]}$ and $N_i:=\iso{N}{\gamma_i(N)}$, $i=1,\dots, s$, with $s$ the nilpotency class of $N$. We refer to the \tf filtration
$$1\nor N_s\nor\cdots\nor N_1\nor K \nor \PI$$
as the \emph{filtration corresponding to $K\nor_f\PI$.} The induced linearisation is called the \emph{linearisation corresponding to $K\nor_f\PI$}. 

Finally, let $\p:\PI\to\PI$ be an endomorphism such that $\p(K)\subseteq K$. Denote by $\p'$ the induced endomorphism on $K$. 
We call the linearisation of $\p'$ (as defined above Theorem~\ref{prod}) also  the \emph{linearisation of $\p$ with respect to $K$}.
\end{definition}
 
\begin{cor}\label{formc}
Let $f$ be a \sf on an infra-solvmanifold $M=\tilde{M}/\PI$, and let $K$ be a fully invariant, finite index \NR subgroup of $\PI$.
Let $\{\Lambda_i, A_i\}_{i=0,\dots, c}$ be the linearisation corresponding to $K\nor_f\PI$, 
and suppose that $f$ induces the endomorphism $\p$ on $\PI$ with linearisation $\{F_0,\dots, F_c\}$ with respect to $K$. Then
  $$ N(f)=\dfrac{1}{[\PI:K]}\sum_{\bar\alpha\in\PI/K}\prod_{i=0}^c \,\di{A_i(\alpha)F_i}.$$
\end{cor}

\begin{proof}
As $f$ induces the morphism $\p$, there exists a lift $\tilde{f}_0:\tilde{M} \to \tilde{M}$ of $f$ such that 
$\p(\gamma) \tilde{f}_0 = \tilde{f}_0 \gamma$ for all $\gamma\in \PI$. 
Let $\bar f$ be the induced lift on $S$.
Let $\bar{\alpha} \in \PI/K$ (so $\alpha \in \PI$). Then $\alpha \tilde{f}_0$ is a lift of $\bar{\alpha} \bar{f}$. It is easy to see that $\alpha \tilde{f}_0 \gamma = \tau_\alpha(\p(\gamma)) \alpha \tilde{f}_0$
for all $\gamma\in \PI$, where
$\tau_\alpha:K\to K:x\mapsto \alpha x\alpha^{-1}$. So $\bar{\alpha} \bar{f}$ induces the endomorphism $\tau_\alpha \circ \p'$ on $K$. 
The automorphism $\tau_\alpha$ induces the automorphism $A_i(\alpha)$ on $\Lambda_i$. 
Hence $\{A_0(\alpha)F_0,\ldots,A_c(\alpha)F_c\}$ is the linearisation of $\tau_\alpha\circ\p'$. 
The result thus follows immediately from Theorem~\ref{prod}.
\end{proof}

\begin{example}\label{big example}
Let $\Pi=\Z^5 \rtimes \Z$ where the generator of the $\Z$-factor is acting on $\Z^5$ via the matrix 
\[ A= \begin{pmatrix}
-1 & 0 & 0 & 0 & 0 \\
 0 & 0 & 1 & 0 & 0 \\
 0 & 0 & 0 & 1 & 0 \\
 0 & 0 & 0 & 0 & 1 \\
 0 &-1 & 1 & 1 & 1
\end{pmatrix}.
\]
This means that elements of $\Pi$ can be seen as tuples $(\vec{z}_2, z_1)$ where $\vec{z}_2\in \Z^5$ (a column vector) and $z_1\in \Z$ and where the product is given by 
\[ (\vec{z}_2, z_1) \cdot ( \vec{z}_2\sss{}' , z_1' ) = ( \vec{z}_2 + A^{z_1} \vec{z}_2\sss{}' , z_1 + z_1').\]
The $4\times 4$ block in the right bottom corner of $A$ was already used in \cite{dlr} and so we know that the eigenvalues of $A$ are 
\[ r_1, r_2 = 
\frac{1 + \sqrt{13} \pm \sqrt{2\sqrt{13} -2}}{4},\; r_3,r_4= \frac{
1 - \sqrt{13} \pm i \sqrt{2\sqrt{13} + 2}}{4},\; r_5= -1,\]
where $|r_3|=|r_4|=1$, but $r_3$ and $r_4$ are not roots of unity. From this it is easily seen that $A^{2z}$ does not have any non-trivial roots of unity as eigenvalues and that the group $K= \Z^5 \rtimes (2\Z)$ is a fully invariant \NR subgroup of $\Pi$, while $\Pi$ itself is not an \NR group. The  linearisation $\{ \Lambda_i, A_i\}_{i=0,1}$ corresponding to $K$ is based on the torsion free filtration
\[ \Pi_2=1 \nor \Pi_1=\sqrt{[K,K]}=\Z^5 \nor \Pi_0 = K \nor_f \Pi\]
and so 
\[ \Lambda_1 = \Z^5, \; \Lambda_0\cong \Z , \mbox{ and} \]
\[ A_1: \Pi \to \Aut{\Z^5}: (\vec{z}_2, z_1) \mapsto A^{z_1} ,\;\; A_0: \Pi \to \Aut{\Z}:  (\vec{z}_2, z_1) \mapsto 1 \]
For any $k\in \Z$ we now define the matrix 
\[ B_k = \begin{pmatrix}
k & 0 & 0 & 0 & 0\\
0 & -1 & 1 & 1 & 0\\
0 & 0 & 0 & -1 & 1\\
0 & 0 & -1 & 1 & 0 \\
0 & -1 & 1 & 0 & 0
\end{pmatrix}.\]
One can check that $B_kA = A^{-1} B_k$ and form this it follows that 
$\varphi_k : \Pi \to \Pi: (\vec{z}_2 , z_1) \mapsto ( B_k \vec{z}_2, - z_1)$ is an endomorphism of $\Pi$ (which is even an automorphism for $k=\pm 1$). \\
Now let $M$ be the solvmanifold with fundamental group $\Pi$. Then $M$ is 2-fold covered by the \NR solvmanifold with fundamental group $K$. Let $f_k$ be a selfmap of $M$ inducing the endomorphism $\varphi_k$. The linearisation of $\varphi_k$ is $\{ F_0 =-1, \,F_1= B_k\}$ and so by the averaging formula from Corollary~\ref{formc} above we find that 
\begin{eqnarray*}
N(f_k) & = & \dfrac{1}{[\PI:K]}\sum_{\bar\alpha\in\PI/K}\prod_{i=0}^1 \,\di{A_i(\alpha)F_i}\\
           & = & \dfrac{1}{2}(   |1-(-1)|\di{B_k}      +     |1-(-1)|\di{A B_k}\\
           & = & \di{B_k} + \di{AB_k} \\
           & = & 3 |1-k| + 3 |1+k| \\
           & = & \left\{ \begin{array}{l}
                              6 \mbox{ if } k =0\\
                              6|k| \mbox{ if } k\neq 0. \end{array}\right. 
\end{eqnarray*}
\end{example}

\section{Averaging formula over a non-invariant subgroup.}\label{ni}

In Corollary~\ref{formc}, we would like to lift the assumption that $K$ is fully invariant. We leave open whether such a result holds in general. In this section, we instead offer an averaging formula in case $K$ is ``net'', in the sense as defined below.

First, recall that $A\in\GL n$ is called \emph{net} if the multiplicative subgroup of $\C^*$ generated by all eigenvalues of $A$ does not contain a nontrivial root of unity. 
We now define:

\begin{definition}
Let $\PI$ be a \tf \vp group, and let $\Gamma\nor_f\PI$ be a normal, finite index \s subgroup of $\PI$. 
Let $\{\Delta_i, \lambda_i\}_{i=0,\dots, c}$ denote the linearisation corresponding to $\Gamma\nor_f\PI$. 
We say that $\Gamma$ is \emph{net} if for every $i\in \{0,\ldots,c\}$ and for all $\gamma$ in $\Gamma$ the automorphism $\lambda_i(\gamma)$ is net.
\end{definition}

We establish:

\begin{theorem}\label{prodni}
Let $f$ be a \sf on an \infra manifold $M=\tilde{M}/\PI$, and let $\KK\nor_f\PI$ be a net normal subgroup of $\PI$.
Let $\K\nor_f \KK$ be a fully invariant, finite index subgroup of $\PI$. 
Suppose that 
\begin{itemize}
\item $f$ induces $\p\in\End{\PI}$ having linearisation $\{F_0,\dots, F_c\}$ with respect to $\K$;
\item $\{\Lambda_i, A_i\}_{i=0,\dots, c}$ is the linearisation corresponding to $\K\nor_f\PI$.
\end{itemize}
Then $$N(f)=\dfrac{1}{[\PI:K']}\sum_{\bar\alpha\in\PI/K'}\prod_{i=0}^c \,\di{A_i(\alpha)F_i}.$$
\end{theorem}

\begin{remark}
As $K'$ is net, any fully invariant subgroup $K\leq_f K'$ will itself be net. Note that given $K'$, we can e.g.\ take $K$ to be the subgroup of $\Gamma$ which is generated by all elements $g^m$ with $g\in \Pi$ and $m=[\Pi:K']$.
\end{remark}

\begin{example}
Before we start with the proof of this theorem, let us give an example of groups $\Pi$ and $K'$ where this theorem is applicable. 
Consider the group $\Z^5 \rtimes \Z$ form Example~\ref{big example} and now take $\Pi= \Z \times (\Z^5\rtimes \Z) $. 
So elements of $\Pi$ can now be written as triples $(z_3, \vec{z}_2, z_1)$. The group $K' = \Z \times (\Z^5 \times 2\Z)  $ is a net normal subgroup of $\Pi$ (of index 2).
However, $K'$ is not fully invariant as the endomorphism $\varphi : \Pi \to \Pi: (z_3, \vec{z}_2, z_1) \mapsto (0, \vec{0},z_3)$ does not map $K'$ to itself. The group 
$K=2\Z \times (\Z^5 \rtimes 2\Z) $ can be used as a fully invariant  finite index subgroup of $K'$ as in the stament of the theorem above.
 
\end{example}

We need three technical lemmas to prove the above theorem.

\begin{lemma}\label{s=c}
Let $\Gamma$ be an \NR group and $\Gamma'\leq_f\Gamma$ a finite index subgroup. 
Then $[\Gamma',\Gamma']\leq_f[\Gamma,\Gamma]$ as well. 
\end{lemma}

\begin{proof}
First, we take a finite index subgroup $S\leq_f \Gamma'$ that is characteristic in $\Gamma.$ This is always possible: consider $\Gamma_0= \bigcap_{\gamma\in \Gamma} \gamma \Gamma' \gamma^{-1}$. As $\gamma' \Gamma' \gamma'^{-1} = \Gamma'$ for all $\gamma'\in \Gamma'$, there are only finitely many different terms $\gamma \Gamma \gamma^{-1}$. Hence, $\Gamma_0$ is a normal subgroup of finite index, say $m$, in $\Gamma$, which is contained in $\Gamma'$. Taking $S$ to be the group generated by all elements of the form $\gamma^m$, we find a finite index characteristic subgroup of $\Gamma$ with $S\leq \Gamma'$.

Since $[S,S]\subseteq[\Gamma',\Gamma']\subseteq[\Gamma,\Gamma]$, it is sufficient to show that $[S,S]\leq_f[\Gamma,\Gamma]$.
To this end, note that $$\frac{S}{[S,S]}\nor_f\frac{\Gamma}{[S,S]},$$ hence $\Gamma/[S,S]$ is virtually abelian. 
Let $G$ be the maximal finite normal subgroup of $\Gamma/[S,S]$, and set $\tilde S:=p^{-1}(G)$, where $p:\Gamma\to\Gamma/[S,S]$ denotes the natural projection. 
Then $\tilde S$ is a characteristic subgroup of $\Gamma $ containing $[S,S]$ as a finite index subgroup.

By construction,  $\Gamma/\tilde S$ has no nontrivial finite normal subgroups. 
As $\Gamma/\tilde S\cong (\Gamma/[S,S])/(\tilde S/[S,S])$ is virtually abelian, 
\cite[Theorem~1.1]{di93-1} implies that $\Gamma/\tilde S $ is a crystallographic group. This means that $\Gamma/\tilde S $ contains a unique maximal abelian and normal subgroup $T/\tilde S$ which is free abelian and of finite index in $\Gamma/\tilde S$.
So  $\Gamma/\tilde S$ fits in the exact sequence
$$1\longrightarrow \Z^k\cong T/\tilde S\longrightarrow\Gamma/\tilde S\overset{q}{\longrightarrow} F\longrightarrow1$$ with $F$ finite.

As $T/\tilde S$ is maximal abelian in $\Gamma/\tilde S$, the finite group $F=\Gamma/T$ acts faithfully on $T/\tilde S$ via conjugation in $\Gamma/\tilde S.$ 

Let $1\nor \tilde S_s\nor\cdots\nor \tilde S$ be a series of characteric subgroups of $\tilde S$  having finite or abelian factors. 
As $\tilde S$ is itself characteristic in $\Gamma$, the extended series
$$1\nor \tilde S_s\nor \cdots\nor\tilde S\nor T\nor \Gamma$$ is a normal series of $\Gamma$ with finite or abelian factors. 
As $\Gamma$ is $\mathcal{NR}$, the action of any element of $\Gamma$ on $T/\tilde S$ by conjugation does not have a nontrivial root of unity as an eigenvalue, 
hence neither does the action of any element of $F\cong\Gamma/T$ on $T/\tilde S$. 
However, as $F$ is finite, the action of $F$ on $T/\tilde S$ must then be trivial, implying $F=1$ as this action is faithful as well. 
Hence $\Gamma/\tilde S=T/\tilde S$ is abelian, showing $[\Gamma,\Gamma]\leq \tilde S$.
So $$[S,S]\leq [\Gamma,\Gamma]\leq \tilde S.$$
Since $[S,S]\leq_f\tilde S$, also $[S,S]\leq_f[\Gamma,\Gamma]$.
\end{proof}

\begin{lemma}\label{tech2}
Let $\PI$ be a torsion free polycyclic-by-finite group, and let $\K, \KK\nor \PI$ be finite index \NR groups with $\K\subseteq \KK$.  
Let $$1\nor N_c\nor \dots\nor N_1=N\nor K\nor_f\PI\ \text{ and }\ 1\nor N'_s\nor \dots\nor N'_1=N'\nor K'\nor_f\PI$$ be the filtrations corresponding to $K\nor_f\PI$ and $K'\nor_f\PI$, respectively. 
Let $\{\Lambda_i, A_i\}_{i=0,\dots, c}$ be the linearisation corresponding to $\K\nor_f \PI$.
Then 
\begin{itemize}
    \item $c=s$;
    \item $N_i=\K\cap N'_i$
    for all $i=1,\dots,c$;
    \item $A_i$ is trivial on $N'$ for all $i=1,\dots,c$.    
\end{itemize}
\end{lemma}
\begin{proof}
As $N\leq_f N'$ by Lemma~\ref{s=c}, also $N_i\leq_f N_i'$ and hence $c=s$.
Moreover, as $\frac{K\cap N'}{N}$ is both torsion free (since $K/N$ is torsion free) and finite (since $N'/N$ is finite), $K\cap N'=N$.
Note that $N/N_i$ is torsion free by definition of $N_i$. Hence, for all $i=1,\dots, c$, also $\frac{K\cap N_i'}{N_i}$ is both torsion free and finite, so $K\cap N_i'=N_i$.
For all $n'\in N'$ and $n_i\in N_i$, it then follows that $$n'n_in'^{-1}n_i^{-1}\in [N',N'_i]\cap N_i\subseteq N'_{i+1}\cap K=N_{i+1},$$
so $A_i$ is trivial on $N'$.
\end{proof}

In order to show that \NR manifolds satisfy the Anosov relation, Keppelmann and McCord \cite[Theorem~4.2]{km95-1} proved the following lemma for the case $k=1$, albeit under the more general assumption that each $A(v)$ does not have nontrivial roots of unity as an eigenvalue.  
We defer the rather lengthy proof of this lemma to the Appendix.

\begin{lemma}\label{matrixanalysis}
Let $X\in\Z^{n\times n}$ and $\Phi\in \Q^{m\times m}$ be matrices and let $A:\Z^m\to\operatorname{SL}_n(\Z)$ be an endomorphism such that $A(v)$ is net for all $v\in\Z^m$.
Suppose that $\Phi$ does not have $1$ as an eigenvalue, and that there exists $k\in\N$ such that $\Phi(k\Z^m)\subseteq \Z^m$ and $XA(kv)=A(\Phi(kv))X$ for all $v\in\Z^m$. 
Then $\de{A(v)X}=\de{X}$ for all $v\in\Z^m$. 
\end{lemma}

We are finally ready for

\begin{proof}[Proof of Theorem~\ref{prodni}]
It is sufficient to prove that the function $\PI\to\R:\alpha\mapsto\prod_{i=0}^c \,\di{A_i(\alpha)F_i}$ is constant on cosets of $K'$.
We will prove this by showing that for each $x\in K'$ and every $\alpha\in\PI$,
\begin{itemize}
    \item $\de{A_0(x)A_0(\alpha)F_0}=\de{A_0(\alpha)F_0}$;
    \item $\de{A_i(x)A_i(\alpha)F_i}=\de{A_i(\alpha)F_i}$ for all $i\in\{1,\ldots,c\}$\\ whenever $\de{A_0(\alpha)F_0}\neq 0$.
\end{itemize}

For the first item, it suffices to note that $K/N$ embeds naturally into $K'/N'$, for $K\cap N'=N$ by Lemma~\ref{tech2}. 
Hence, if $x\in K'$, the commutative diagram
\begin{center}
\begin{tikzcd}[row sep=large]
K/N \arrow[hook]{r}{} \arrow{d}[swap]{A_0(x)\,}
               &K'/N' \arrow{d}{\,\mathrm{Id}}\\
K/N \arrow[hook]{r}{} &K'/N'
\end{tikzcd}
\end{center}
shows that $A_0(x)=I$.

For the second item, take $x\in K'$, $\alpha\in\PI$ and $i\in\{1,\dots,c\}$.
It is easy to verify that 
$$\tau_\alpha\circ\p\circ\tau_k=\tau_{(\tau_\alpha\circ\p)(k)}\circ\tau_\alpha\circ\p$$ for all $k\in K.$ 
Comparing the maps induced on $\Lambda_i$, we get
\begin{equation*}
    A_i(\alpha)F_i\rho_i(\bar k)=\rho_i(A_0(\alpha)F_0(\bar k))A_i(\alpha)F_i\tag{$\ast$}
\end{equation*}
for all $\bar k\in K/N$. Here $\rho_i:K/N\to\Aut{\Lambda_i}$ is the conjugacy action as defined on page~\pageref{rho_i}.

Identify $K'/N'$ with $\Z^m$ for some $m\in\N$. Then $K'/N'$ sits naturally in $\Q^m$. 
As $K/N$ is a finite index subgroup of $K'/N'$, there exists $\Phi\in\Q^{m\times m}$ making the diagram
\begin{center}
\begin{tikzcd}[row sep=large]
K/N \arrow[hook]{r}{} \arrow{d}[swap]{A_0(\alpha)F_0\,}
               &K'/N'\cong\Z^m \arrow[hook]{r}{} &\Q^m \arrow{d}{\,\Phi}\\
K/N \arrow[hook]{r}{} &K'/N'\cong\Z^m \arrow[hook]{r}{} &\Q^m
\end{tikzcd}
\end{center}
commute. 
Note that $\Phi(d\Z^m)\subseteq\Z^m$ for any $d\in \N$ satisfying $d\Z^m\subseteq K/N$ (in
$K'/N'\cong\Z^m$).

As $A_i$ is trivial on $N'$ by Lemma~\ref{tech2}, $\rho_i$ extends to an action
$\tilde\rho_i:K'/N'\to \Aut{\Lambda_i}$. 
Using $K'/N'\cong \Z^m$ and $\Lambda_i\cong \Z^n$ for some $n\in\Z$, we obtain a morphism $A:\Z^m\to \SLZ n.$

Let $X\in \Z^{n\times n}$ represent $A_i(\alpha)F_i$ on $\Lambda_i\cong\Z^n$. 
Then ($\ast$) implies that $$XA(dv)=A(\Phi(dv))X$$  for all $v\in\Z^m
.$
Lemma~\ref{matrixanalysis} now asserts that 
$\det(I-A(v)X)=\det(I-X)$ for all $v\in \Z^m$  when $\de{\Phi}\neq 0$, or, by translating back, that 
$\det(I-A_i(x)A_i(\alpha) F_i) =\det(I-A_i(\alpha) F_i) )$ for all $x\in K'$ whenever  $\de{A_0(\alpha)F_0}\neq 0$.
\end{proof}

The expression for $N(f)$ in Theorem~\ref{prodni} is not completely satisfactory since it still depends on the fully invariant \NR subgroup $K$. We next explain how to compute this expression directly from $K'$, without having to know $K$ explicitly.

Let $K'\nor_f \PI$ be net, say $[\PI:K']=m\in\N$. Consider the subgroup $K:=\langle\{g^m\mid g\in\PI\}\rangle$ generated by all the $m$-th powers of elements of $\PI$. Then $K\nor_f K'$ is a fully invariant finite index subgroup of $\Pi$.

Let $1\nor N_c\nor \dots\nor N_1=N\nor K\nor_f\PI$ and $1\nor N'_c\nor \dots\nor N'_1=N'\nor K'\nor_f\PI$ be the filtrations corresponding to $K\nor_f\PI$ and $K'\nor_f\PI$, respectively. 
Then, using the same notations as before,
$\tau_\alpha\circ\p$ induces the morphisms 
$A_i(\alpha)F_i$, $i=0,\dots,c$, on $N_i/N_{i+1}$.
As $N_i/N_{i+1}$ embeds as a finite index subgroup of $N'_i/N'_{i+1}$, there exists $M^i_{\tau_\alpha\circ\p}\in\Q^{k_i\times k_i}$ fitting in the commutative diagram
\begin{center}
\begin{tikzcd}[row sep=large]
N_i/N_{i+1} \arrow[hook]{r}{} \arrow{d}[swap]{A_i(\alpha)F_i\,}
               &N'_i/N'_{i+1}\cong\Z^{k_i} \arrow[hook]{r}{} &\Q^m \arrow{d}{\,M^i_{\tau_\alpha\circ\p}}\\
N_i/N_{i+1} \arrow[hook]{r}{} &N'_i/N_{i+1}'\cong\Z^{k_i} \arrow[hook]{r}{} &\Q^m
\end{tikzcd}
\end{center}
Moreover, $\de{A_i(\alpha)F_i}=\de{M^i_{\tau_\alpha\circ\p}}$.

Since $m\Z^{k_i}\subseteq N_i/N_{i+1}$ (in $N'_i/N'_{i+1}\cong\Z^{k_i}$), we can compute $M^i_{\tau_\alpha\circ\p}$ explicitly by fixing an isomorphism $N_i'/N_{i+1}'\cong \Z^{k_i}.$ Indeed, take $z_j\in N_i'$ such that $\{\bar z_1,\dots,\bar z_{k_i}\}$ generates $N_i'/N_{i+1}'.$ Let $\{e_1,\dots, e_{k_i}\}$ denote the standard basis of $\Z^{k_i}$. Then $N'_i/N'_{i+1}\cong\Z^{k_i}$ via $\bar z_j\leftrightarrow e_j$.

By construction,  $z_j^m \in N_i$, hence also $\tau_\alpha\circ\p(z_j^m)\in N_i$. Therefore, we can  write $\tau_\alpha\circ\p(z_j^m)N_{i+1}'$ uniquely as $$\tau_\alpha\circ\p(z_j^m)N_{i+1}'=\lambda_{1j}\bar z_1+\dots +\lambda_{k_i j}\bar z_{k_i}\in \dfrac{N_i'}{N_{i+1}'}$$ with $\lambda_{l j}\in\Z$.

From the commutativity of the above diagram, $M^i_{\tau_\alpha\circ\p}(me_j)=(\lambda_{1j},\dots,\lambda_{k_i j})$ for all $j=1,\dots, k_i$, hence
$$M^i_{\tau_\alpha\circ\p}=\frac{1}{m}
\begin{pmatrix}
\lambda_{11} & \dots & \lambda_{1 k_i}\\
\vdots & \ddots & \vdots\\
\lambda_{k_i 1} & \dots & \lambda_{k_i k_i}
\end{pmatrix}.$$

By the above reasoning, we conclude with the following
\begin{cor}
Let $f$ be a \sf on an \infra manifold $\R^h/\PI$ inducing an endomorphism $\p$ on $\Pi$. Suppose that $K'\nor_f\PI$ is a net, normal and finite index subgroup of $\PI$. 
Then
$$N(f)=\dfrac{1}{[\PI:K']}\sum_{\bar\alpha\in\PI/K'}\prod_{i=0}^c \,\di{M^i_{\tau_\alpha\circ\p}}$$
with $M^i_{\tau_\alpha\circ\p}$ as defined above.
\end{cor}

\section{Infra-solvmanifolds of type (R)}
In this section, we show how to deduce the known averaging formula \cite[Theorem 4.3]{ll09-1} for Nielsen numbers on infra-solvmanifolds of type (R) (and so also on infra-nilmanifolds) from our general formula.

A simply connected solvable Lie group $G$ is said to be of type (R) if for every $X\in \Lg$, the corresponding Lie algebra of $G$,
the inner derivation $\ad(X)$ only has real eigenvalues.

The affine group $\Aff(G)$ of a solvable Lie group $G$ is the semidirect product $\Aff(G)=G\rtimes \Aut{G}$. 
It embeds naturally in the semigroup $\aff(G)=G\rtimes\End{G}$ consisting of all pairs $(d,D)$ with $d\in G$ and $D\in \End{G}$ an endomorphism of $G$.
The product in $\aff(G)$ (and $\Aff(G)$) is given by $(d,D)(e,E)= (dD(e), DE)$. 
Both $\aff(G)$ and $\Aff(G)$ act on $G$ via $(d,D)\cdot g = d D(g)$.

\begin{definition} \label{typeR}
An \emph{infra-solvmanifold of type (R)} is a quotient manifold of the form $G/\Pi$ where $G$ is a simply connected solvable Lie group of type (R),
$\Pi\subseteq \Aff(G)$ is a torsion free subgroup of the affine group of $G$ such that $\Gamma=G\cap \Pi$ is of finite index in $\Pi$ and 
$\Gamma$ is a discrete and cocompact subgoup of $G$.  
The finite quotient $\Psi=\Pi/\Gamma$ is called the \emph{holonomy group} of $\Pi$.
\end{definition}

\begin{remark}
We can view $G$ as a normal subgroup of $\Aff(G)$, so that it makes sense to talk about the intersection $G\cap \Pi$. It is easy to see that $\Gamma=G \cap \Pi=\{(a,A)\in \Pi\mid A=I\}$ and $\Psi\cong\{ A\in \Aut{G}\mid \exists a \in G:\; (a,A)\in \Pi\}$.
\end{remark}

\begin{remark}
Definition~\ref{typeR} implies that $\Pi$ is acting properly discontinuously and cocompactly on $G$ so that $G/\Pi$ is indeed an infra-solvmanifold. 
\end{remark}

For the rest of this section we assume that $G$, $\Pi$, $\Gamma$ and $\Psi$ are as in the definition above.

The group $\Gamma$ is net (see e.g.\ \cite[Corollary 3.11]{deki95-1}) and hence the manifold $G/\Gamma$ is an \NR solvmanifold

Consider the subgroup $K$ of $\Pi$ which is generated by all elements of the form $(a,A)^m$, where $(a,A)\in \Pi$ and $m=|\Psi|=[\Pi:\Gamma]$. Then $K\nor_f \Gamma$ is a fully invariant subgroup of finite index in $\Pi$.
Note that $K$, being a finite index subgroup of $\Gamma$, is also a discrete and cocompact subgroup of $G$. Hence any endomorphism (resp.\ automorphism) $\varphi$ of $K$ extends uniquely to an endomorphism (resp.~automorphism) $\tilde{\varphi}$ of $G$ (see e.g.~\cite{gorb75-1}).

Let $f:G/\Pi\to G/\Pi$ be a self-map.
Assume that $f$ induces the morphism $\varphi:\Pi\to \Pi$. Then there exists a lift $\tilde{f}:G\to G$ of $f$ such that 
\[ \forall g\in \Pi:\; \varphi(g) \circ \tilde{f} = 
\tilde{f} \circ g.\]
By \cite[Theorem 2.2]{ll09-1} there exists $(d,D)\in \aff(G)$ satisfying 
\begin{equation}\label{leelee}
 \forall g\in \Pi:\; \varphi(g) \circ (d,D) = 
(d,D) \circ g.\end{equation}
Although we will not really need this fact, we want to mention that this implies that $(d,D)$ induces a map
\[
 \overline{(d,D)}:G/\Pi \to G/\Pi:
[ g ]  \to [(d,D)\cdot g],
\]
(where $[g]=\Pi \cdot g$ denotes the orbit of $g$ under the action of $\Pi$)
which is homotopic to $f$, and so $N(f)=N(\overline{(d,D)})$.

Let $\varphi'$ denote the restriction of $\varphi$ to $K$, so $\varphi(\gamma,1) = (\varphi'(\gamma), 1)$ (where we identify the element $g\in G$ with the element $(g,1) \in G\rtimes \Aut{G}=\Aff(G)$).
From \eqref{leelee} we find that for all $\gamma \in K$:
\[ \varphi(\gamma,1) (d,D)= (d, D) (\gamma, 1) 
\Rightarrow (\varphi'(\gamma) d , D) = 
(d D(\gamma), D) \Rightarrow
\varphi'(\gamma) = d D(\gamma) d^{-1}. \] 
Let us denote the unique extension of $\varphi'$ to 
$G$ by $\tilde{\varphi}$, then 
obviously $\tilde{\varphi}= \mu(d)D$, where $\mu(d)$ denotes conjugation with $d$.

Now, let $\{F_0, F_1, \ldots, F_c\}$ be the linearisation of $\varphi$ with respect to $K$ and let 
$\{\Lambda_i, A_i\}$ be the linearisation corresponding to $K\nor_f \Pi$. Theorem~\ref{prodni} implies that 
\[N(f) = \frac{1}{|\Psi|}\sum_{\overline{\alpha}\in \Psi} \prod_{i=0}^c \di{ A_i(\alpha ) F_i}. \]

In the formula above $A_i(\alpha) F_i$ is the endomorphism on a free abelian factor $\Z^{k_i}$ (which is either of the form $K/N$ in case $i=0$ or 
of the form $N_i/N_{i+1}$ when $i\in \{1,2,\ldots, c\}$) and is induced by the endomorphism $\mu(\alpha) \circ \varphi'$.

Let us denote the eigenvalues of $A_i(\alpha) F_i $ by $\mu_{i,1}, \mu_{i,2}, \ldots, \mu_{i,k_i}$ (where we list each eigenvalue as many times as its multiplicity). Then 
\[ \prod_{i=0}^c \di{ A_i(\alpha ) F_i} = \prod_{i=0}^c \prod_{j=1}^{k_i} | 1 -  \mu_{i,j}|.\] 

Assume that $\alpha= (a,A) \in \Aff(G)$, then we have for all $\gamma \in K$:
\[ (\mu(\alpha)(\gamma),1) = (a, A) (\gamma,1) (a,A)^{-1} = (a A(\gamma) a^{-1}, 1).\]
From this it follows that the unique extension of $\mu(\alpha)$ to the Lie group $G$ equals $\mu(a) \circ A$ and 
combining this with what we already knew for $\varphi'$ we know that the unique extension of $\mu(\alpha) \circ \varphi$ to $G$ is 
$\mu(a) \circ A \circ \mu(d) \circ D$.
Using $\beta_\ast:\Lg \to \Lg $ to denote the differential of an endomorphism $\beta:G\to G$, the 
collection of eigenvalues $\mu_{i,j}$ is exactly the collection of the eigenvalues of the linear map 
$(\mu(a) \circ A \circ \mu(d) \circ D)_\ast$ on $\Lg$. (This was proven in detail in \cite[Lemma 3.2 and Proposition 3.9]{deki95-1} in case 
$\mu(\alpha) \circ \varphi'$ is an automorphism, but the proof works for endomorphisms too.) 
So it follows that (using that $\mu(x)_\ast =\Ad(x)$ for all $x\in G$):
\[ \prod_{i=0}^c \di{A_i(\alpha ) F_i} = \prod_{i=0}^c \prod_{j=1}^{k_i} | 1 -  \mu_{i,j}|=
\di{\Ad(a)\, A_\ast\, \Ad(d)\,  D_\ast}.\] 
In \cite[Theorem 1]{hl16-1} it was shown that for any endormorphism $B:G\to G$ and any $x\in G$ the equality \[ \det (I - \Ad(x)\, B_\ast ) = 
\det(I - B_\ast) \]
holds. We use this and the fact that $A_\ast$ is invertible (since $A\in \Aut{G}$) 
to rewrite $\det(I - \Ad(a)\, A_\ast\, \Ad(d)\,  D_\ast)$:
\begin{align*}
    \det(I - \Ad(a)\, A_\ast\, \Ad(d)\,  D_\ast)\  & = 
\ \det(I - A_\ast\, \Ad(d)\,  D_\ast) \\
& =  \ \det( I - \Ad(d)\,  D_\ast\, A_\ast) \\
& =  \ \det( I -  D_\ast\, A_\ast) \\
& =  \ \det( I -  A_\ast\, D_\ast )
\end{align*}
We conclude that 
\[ N(f)= \frac{1}{|\Psi|}\sum_{A\in \Psi}\di{A_\ast\, D_\ast}.\]

Note that $\lvert\det(A_\ast)\rvert=1$ as $A$ is an automorphism of finite order, hence $\lvert\det(I-A_\ast D_\ast)|= 
\lvert\det(A_\ast^{-1} - D_\ast)| $. This allows us to rewrite the formula above also as 
\[ N(f)= \frac{1}{|\Psi|}\sum_{A\in \Psi}\lvert \det(A_\ast -  D_\ast)\rvert\]
which is exactly the same formula as in \cite[Theorem 4.3]{ll09-1}.
\section{Polynomial maps}\label{pol}

In this section we consider the situation in which the infra-solvmanifold is represented as 
a special kind of polynomial manifold $M=\R^h/\Pi$, which we call of canonical type (see \cite{dil94-1,di96-1}).
As mentioned in \cite{dere18-1}, when $M$ is a polynomial manifold,  every \sf $f:M\to M$ is homotopic to a polynomial map $\bar p$. By this we mean a map $\bar p$ whose lift $p:\R^h \to \R^h$ to the universal covering of $M$ is a polynomial map.
In this section, we express the averaging formula found in Section~\ref{ni} in terms of the polynomial $p$.

\subsection{Canonical type representations}
First,
we fix the following notations. Let $\PI_*: 1=\PI_{n+1}\nor \PI_n\nor \cdots\nor \PI_1\nor_f \PI$ be a \tf filtration, and let  $\{\Lm i, \lm i\}$ be the corresponding linearisation.
Then $\Lm i\cong \Z^{k_i}$ for some $k_i\in\N$; fix an isomorphism $j_i: \Lm i \to \Z^{k_i}$. Via this isomorphism, 
$\lambda_i(\gamma) \in \Aut{\Z^{k_i}}=\GLZ{k_i}$.

We set $K_i:=k_1+\cdots+k_i$ for $i=1,\dots, n$, and $K_l:=0$ for $l\leq 0$. Note that $K_n=h$, the Hirsch length of $\PI$.
We further identify $\R^h$ with $\Rk 1\times \dots\times \Rk n$, so we write $x\in\R^h$ as $(x_1,\ldots, x_n)$ with $x_i\in\Rk i$. 
For $x\in\R^h$, we also use the notation $\x i:=(x_1,\dots,x_i)$ to denote the projection of $x$ to $\RK i$. For $L,l$ nonnegative integers, we let $\pp{L}{l}$ denote the set of polynomial maps from $\R^L$ to $\R^l$ and with P$(\R^L)$ we denote the group of polynomial diffeomorphisms 
of $\R^L$, see Subsection~\ref{polymani}.

\begin{theorem}[Dekimpe-Igodt {\cite[Theorem 4.1]{di96-1}}]\label{thmcan}
Let $\PI$ be a \tf \vp group with \tf filtration $\PI_*$. Then, with the notations introduced
above, there exists a representation $\rho:\PI\to\pd h$ satisfying the following properties:
\begin{enumerate}
\item For all $\gamma\in \Pi$ and for all  $i\in\{1,\dots,n\}$, there exists $q_i\in\pp{K_{i-1}}{k_i}$ such that 
$$ \rho(\gamma):\R^h\to \R^h: 
\left(
\begin{array}{c}
x_1\\ x_2\\ \vdots \\ x_n
\end{array}
\right) \mapsto
\left(
\arraycolsep=0.9pt
\begin{array}{rcl}
\lambda_1(\gamma)x_1& + &q_1\\
\lambda_2(\gamma)x_2& + &q_2(x_1)\\
&\vdots & \\
\lambda_n(\gamma)x_n& + & q_n(x_1,x_2, \ldots, q_{n-1})
\end{array}
\right)
.$$

\item \label{2} Moreover if $\gamma\in \PI_i$ and $j_i(\gamma\PI_{i+1})=z\in \Z^{k_i}$
then the above form specialises to 
$$ \rho(\gamma):\R^h\to \R^h: 
\left(
\begin{array}{c}
x_1\\ x_2\\ \vdots \\ x_n
\end{array}
\right) \mapsto
\left(\arraycolsep=0.9pt
\begin{array}{rcl}
& x_1&  \\
&  \vdots&  \\
& \makebox(0,0){$x_{i-1}$} & \\
 x_{i} & + & z \\
\lambda_{i+1}(\gamma) (x_{i+1}) & + & q_{i+1} (x_1, \ldots , x_i) \\ 
&\vdots & \\
\lambda_n(\gamma)x_n& + & q_n(x_1,x_2,\ldots, q_{n-1})
\end{array}
\right)
.$$
\item Via the representation $\rho$, the group $\Pi$ acts properly discontinuously and cocompactly on $\R^h$.
\end{enumerate}
\end{theorem}

A representation $\rho:\PI\to\pd h$ satisfying the above properties is called a \emph{canonical type} polynomial representation with respect to $\PI_\ast$. When $\rho$ is clear from the context, we also write $^\gamma x$ instead of $\rho(\gamma)(x)$.

We next analyse the structure of polynomial maps inducing $\p$ on $\PI$.

\begin{lemma}\label{vormp}
Let $\PI$ be a \tf \vp group with \tf filtration $\PI_*$, and let $\rho:\PI\to\pd h$ be a representation of canonical type with respect to $\PI_*$. 
Let $\p:\PI\to\PI$ be an endomorphism leaving each $\PI_i$ invariant; 
denote the induced endomorphisms on $\Lm i\cong \Z^{k_i}$ by a $k_i\times k_i$ integral matrix $G_i$.
Suppose that $p\in\pp{h}{h}$ satisfies $p\circ \rho(\gamma)=\rho(\p(\gamma))\circ p$ for every $\gamma\in\PI$.
Then, for every $i\in\{1,\dots, n\}$, there exist $p_i\in\pp{K_{i-1}}{k_i}$ such that 
\begin{align*}
{p(x)}_i
&=G_i x_i + p_i(x_1,\dots, x_{i-1})\\
\intertext{for every $x\in\R^h$.
Put differently,}
p\vc{x_1}{x_n}
&=\left(\arraycolsep=0.9pt \begin{array}{rcl}
G_1 x_1& + & p_1\\
G_2 x_2 & + & p_2(x_1)\\
& \vdots & \\
G_n x_n & + & p_n(x_1,\dots, x_{n-1})
\end{array}\right)
\end{align*}
for all $x\in\R^h$.
\end{lemma}

\begin{proof}

We first prove that $p(x)_i$ does not depend on $x_j$ with $j>i$ by proving the following

\stepi{Observation. If $p(x)_i$ does not depend on $x_{j+1},\dots, x_n$ for some $j\in\{i+1,\dots, n\}$, then $p(x)_i$ does not depend on $x_{j}$ either.}

\medskip

By assumption, we can write $p(x)_i=f_i(x_1,x_2, \ldots, x_j)$ with $f_i\in \pp{K_j}{k_i}$.\\
Moreover, Theorem~\ref{thmcan}(\ref{2}) implies that for all $\gamma \in\PI_j$,
\begin{itemize}
    \item $\big(p(^\gamma x)\big)_i=\big({^{\varphi(\gamma)}p(x)\big)_i=p(x)}_i$ since also $\varphi(\gamma)\in \Pi_j$ and $i<j$;
    \item $({^\gamma x})_l=x_l$ for all $l<j$;
    \item $({^\gamma x})_j=x_j+z$, with $z$ corresponding to $\gamma\Pi_{j+i}$ in $\Z^{k_j}\cong  \frac{\Pi_j}{\Pi_{j+1}}$
\end{itemize}
Combined, these facts show that for all $\bar\gamma\in\PI_j/\PI_{j+1}$ and so for all $z \in \Z^{k_j}$,
\begin{align*}
f_i(x_1,\dots, x_{j-1}, x_j+ z)
&=
f_i(x_1,\dots, x_{j-1}, x_j).\\
\intertext{The above equality can be seen as a polynomial identity which is satisfied on $\Z^{k_j}$. Since the only $p\in\pp{L}{l}$ vanishing on $\Z^L$ is the zero polynomial, also }
f_i(x_1,\dots, x_{j-1}, x_j+ r)
&=
f_i(x_1,\dots, x_{j-1}, x_j)
\end{align*}
for every $r\in\Rk j$. We conclude that $p(x)_i$ does not depend on $x_j$.

\medskip

So we already know that $p$ is of the form  
$$p\vc{x_1}{x_n}=\left(\begin{array}{c}
f_1(x_1)\\
f_2(x_1,x_2)\\
\vdots\\
f_n(x_1,x_2,\ldots, x_n)
\end{array}\right),$$
with $f_i\in \pp{K_i}{k_i}$.
We now show that the $f_i$ have the prescribed form of the lemma. 

Take $i\in\{1,\dots, n\}$ and $x\in\R^h$. Define $x'\in\R^h$  by setting $x'_l=x_l$ if $l<i$ and $x'_l:=0$ if $l\geq i$. 

Take $\gamma\in\PI_i$, with $\gamma\PI_{i+1}=z$
in $\PI_i/\PI_{i+1}\cong \Z^{k_i}$.
Then ${\p(\gamma)}\PI_{i+1}=G_iz$ in $\PI_i/\PI_{i+1}\cong \Z^{k_i}$. 
So
\begin{align*}
\big(p(\,{^\gamma x'})\big)_i
&=f_i(\,\xx{({^\gamma x'})}{i})\\
&=f_i(x_1,\dots, x_{i-1}, z),\\
\intertext{while}
\big({^{\p(\gamma)}p(x')}\big)_i
&=p(x')\,_i+G_iz\\
&=f_i(x_1,\dots, x_{i-1},0)+G_iz.
\end{align*}
As $p(\,{^\gamma x'})={^{\p(\gamma)}p(x')}$ for any $\gamma\in\PI_i$, the polynomial relation
$$f_i(x_1,\dots, x_{i-1}, z)=f_i(x_1,\dots, x_{i-1},0)+G_iz$$
is satisfied for any $z\in\Z^{k_i}$.
Hence it is satisfied for  any $z\in\Rk i$. 
The proof now finishes by taking $z:=x_i$ and defining $p_i(r_1,\dots, r_{i-1}):=f_i(r_1,\dots, r_{i-1},0)$ for $(r_1,\dots, r_{i-1})\in\RK{i-1}$. 
\end{proof}

\subsection{Averaging formula}

Let $\PI$ be a \tf \vp group of Hirsch length $h$. 
Consider a canonical type polynomial representation $\rho:\Pi\to \PP$.
Then we say that the corresponding quotient space $M=\R^h/\Pi$ is a \emph{canonical type polynomial manifold} realising the infra-solvmanifold with fundamental group $\Pi$.

\begin{theorem}[Der\'e {\cite[Corollary~6.1]{dere18-1}}]
Let $\PI$ be a \tf \vp group of Hirsch length $h$ and let $M=\R^h/\Pi$ be a canonical type polynomial manifold realising the 
infra-solvmanifold with fundamental group $\Pi$ via a representation $\rho:\Pi\to \PP$. Let $f:M\to M$ be a \sf inducing an endomorphism $\p$ on $\PI$.
Then there exists a polynomial map $p:\R^h\to\R^h$ such that
$p\circ\rho(\gamma)=\rho(\p(\gamma))\circ p$ for every $\gamma$ in $\PI$.
\end{theorem}

This result implies that $p$ induces a \sf $\bar p$ of $M=\R^n/\PI$ fitting in the commutative diagram (with $\R^h \to M=\R^h/\Pi$ the natural projection map):
\begin{center}
\begin{tikzcd}[ampersand replacement=\&]
\R^h \arrow{r}{p} \arrow{d}{}
              \&\R^h \arrow{d}{}\\
M \arrow{r}[swap]{\bar p} \&M,
\end{tikzcd}
\end{center}
and that $\bar p$ induces the endomorphism $\p$ on $\PI$. 
As both $\bar p$ and $f$ induce $\p$ on $\PI$ and $M$ is $K(\Pi,1)$, we conclude that $\bar p$ and $f$ are homotopic. Hence $N(f)=N(\bar p)$.
We will refer to the map $p$ as a \emph{polynomial homotopy lift} of $f$.

Combining Theorem~\ref{thmcan} and Lemma~\ref{vormp}, we can now simplify our averaging formula using this polynomial $p$. 

\begin{theorem}\label{aveforusingp} Let $\Pi$ be a torsion free polycyclic-by-finite group and $\Gamma \nor_f \Pi$ a net normal subgroup of finite index in $\Pi$. 
Let $\rho:\PI\to\pd h$ be a canonical type polynomial representation with respect to 
the filtration corresponding to $\Gamma\nor_f\Pi$; let $M=\R^h/\Pi$ be the corresponding canonical type polynomial manifold realising the infra-solvmanifold with fundamental group $\Pi$. 
Let $f:M\to M$ be a \sf of $M$ and let $p:\R^h\to \R^h$ be a homotopy lift of $f$.
\begin{enumerate}
    \item For any $\gamma\in \Pi$ it holds that $\det(I-J(\rho(\gamma)\circ p)_{x_0})$ is independent of the point $x_0\in \R^h$.
    
    \item The Nielsen number of $f$ equals $$N(f)=\dfrac{1}{[\PI:K]}\sum_{\bar \al\in\PI/K}\lvert\det(I-J(\rho(\alpha)\circ p)_{x_0})|.$$
\end{enumerate} Here we use $J(g)_{x_0}$ to denote the Jacobian matrix of a polynomial map $g$ evaluated in the point $x_0\in \R^h$.
\end{theorem}

\begin{proof}

Let $\varphi:\Pi \to \Pi$ be the endomorphism  induced by the \sf $f$.
Then $p \circ \rho(\gamma) = \rho(\varphi(\gamma)) \circ p$ for all $\gamma\in\PI$ as $p$ is a homotopy lift of $f$.
Choose a fully invariant subgroup $K\subseteq \Gamma$ of $\Pi$ that is of finite index in $\Gamma$. Let $M_1=\sqrt[K]{[K,K]}$ and 
$M_i=\sqrt[M]{\gamma_i(M)}$ ($i\geq 2$).

It is given that $\rho$ is a canonical type polynomial representation with respect to the torsion free filtration 
$$
\Pi_\ast:1 \nor \Pi_{c+1} = N_c \nor \cdots \nor \Pi_3= N_2 \nor \Pi_2 = N_1 \nor \Pi_1= \Gamma \nor_f \Pi
$$ with $N_1= \sqrt[\Gamma]{[\Gamma,\Gamma]}$ and $N_i = \sqrt[N]{\gamma_i(N)}$ ($i\geq 2$).
In the definition of a canonical type polynomial representation, we fixed isomorphims $j_i: \Pi_i/\Pi_{i+1} \to \Z^{k_i}$ and for 
$\gamma\in \Pi_i$ with $j_i(\gamma \Pi_{i+1})=z$ it holds that 

$$ \rho(\gamma):\R^h\to \R^h: 
\left(
\begin{array}{c}
x_1\\ x_2\\ \vdots \\ x_n
\end{array}
\right) \mapsto
\left(\arraycolsep=0.9pt
\begin{array}{rcl}
& x_1&  \\
&  \vdots&  \\
& \makebox(0,0){$x_{i-1}$} & \\
 x_{i} & + & z \\
\lambda_{i+1}(\gamma) (x_{i+1}) & + & q_{i+1} (x_1, \ldots , x_i) \\ 
&\vdots & \\
\lambda_n(\gamma)x_n& + & q_n(x_1,x_2,\ldots, q_{n-1})
\end{array}
\right)
.$$

Consider the torsion free filtration 
$$
\Pi'_\ast:1 \nor \Pi'_{c+1} = M_c \nor \cdots \nor \Pi'_3= M_2 \nor \Pi'_2 = M_1 \nor \Pi'_1= K \nor_f \Pi.
$$
Then $j_i$ induces maps 
$j_i' : \Pi'_i/\Pi'_{i+1} \to \Z^k_i: \gamma \Pi'_{i+1}\mapsto j_i(\gamma)$, which are well defined since $\Pi'_i=K \cap \Pi_i$ by Lemma~\ref{tech2}. It follows that 
$j_i'(\Pi'_i/\Pi'_{i+1})$ is a finite index subgroup of $\Z^{k_i}$. 
Choose a matrix $B_i\in \GLQ{k_i}$ with $B_i (j_i'(\Pi'_i/\Pi'_{i+1}))=\Z^{k_i}$ and let 
$\tilde{j}_i= B_i j_i'$. Consider the linear map $B$ of $\R^h$ which is given by the blocked diagonal matrix 
\[\begin{pmatrix} 
B_1 & 0 & \cdots & 0\\
0 & B_2 & \cdots & 0 \\
\vdots & \vdots & \ddots & \vdots\\
0 & 0 & \cdots & B_{c+1}
\end{pmatrix}.
\]
The representation $\rho'= B \circ \rho \circ B^{-1}$ is then canonical with respect to the torsion free filtration corresponding to $K\nor_f \Pi$ (using the maps $\tilde{j}_i$ for the identification of $\Pi_i/\Pi_{i+1}$ with $\Z^{k_i}$).

By definition of $\rho'$, we have  for any $\gamma \in \Pi$ a commutative diagram
\begin{center}
\begin{tikzcd}[ampersand replacement=\&]
\R^h \arrow{r}{B} \arrow{d}[swap]{\rho(\gamma)}
              \&\R^h \arrow{d}{\rho'(\gamma)}\\
\R^h \arrow{r}[swap]{B} \& \R^h.
\end{tikzcd}
\end{center}
Of course, $B$ induces a diffeomorphism $\bar{B}:M=\R^h/\rho(\Pi)\to M'=\R^h/\rho'(\pi)$.
Take $p'= B \circ p \circ B^{-1}$, so that $p'$ makes the following diagram commutative:
\begin{center}
\begin{tikzcd}[ampersand replacement=\&]
\R^h \arrow{r}{B} \arrow{d}[swap]{p}
              \&\R^h \arrow{d}{p'}\\
\R^h \arrow{r}[swap]{B} \& \R^h.
\end{tikzcd}
\end{center}
As $p' \circ \rho'(\gamma) = \rho'(\varphi(\gamma)) \circ p'$, the map $p'$ induces a map $\bar{p}':M' \to M'$, which also induces the endomorphism $\p$ on $\Pi$. 
Since 
\begin{center}
\begin{tikzcd}[ampersand replacement=\&]
M \arrow{r}{\bar B} \arrow{d}[swap]{\bar p}
              \& M' \arrow{d}{\bar p'}\\
M \arrow{r}[swap]{\bar B} \& M'
\end{tikzcd}
\end{center}
is a commutative diagram, where $\bar B$ is a diffeomorphism, it holds that 
\[ N(f)=N(\bar p)= N(\bar p').\]

Now, let $\{F_0, F_1,\ldots, F_c\}$ denote the linearisation of $\varphi$ with respect to $K$ and $\{\Lambda_i=\Pi_{i+1}'/\Pi_{i+2}', A_i\}_{i=0,\ldots ,c}$ the linearisation corresponding to $K\nor_f \Pi$.
It follows from Theorem~\ref{thmcan} and Lemma~\ref{vormp} that 
$$\prod_{i=0}^c \det( I- A_i(\alpha)F_i)
=\det(I-J(\rho'(\alpha)\circ p')_{x_0})$$
for every $\alpha$ in $\PI$, so this term is independent of the chosen $x_0$. 
Note that 
 \[ \det(I-J(\rho(\alpha)\circ p)_{x_0}) =
 \det(I - J( B^{-1}\circ  \rho'(\alpha) \circ p' \circ B)_{x_0})=\]
 \[ \det(I - B^{-1} J(\rho'(\alpha) \circ p')_{B(x_0)}B) = 
\det(I -  J(\rho'(\alpha) \circ p')_{B(x_0)}) \]
thus also $\det(I-J(\rho(\alpha)\circ p)_{x_0})$ is 
independent of the chosen $x_0$. 

By Theorem~\ref{prodni}, we conclude that 
\begin{align*}
N(f) 
&= \frac{1}{[\Pi:K]}\sum_{\bar\alpha\in \Pi/K}\prod_{i=0}^c \lvert\det( I- A_i(\alpha)F_i)|\\
&=\frac{1}{[\Pi:K]}\sum_{\bar\alpha\in \Pi/K}\lvert\det(I-J(\rho'(\alpha)\circ p')_{x_0})|\\
&=\dfrac{1}{[\PI:K]}\sum_{\bar \al\in\PI/K}| \det(I-J(\rho(\alpha)\circ p)_{x_0})|.
\end{align*}
\end{proof}

\begin{remark}
We want to remark that in the above theorem it is important that we are working with a canonical type polynomial representation. If $
\rho: \Pi \to \pd h$ is a general representation (but still with polynomials of bounded degree) letting $\Pi$ act properly discontinuously and cocompactly on $\R^h$ so that $M=\R^h/\rho(\Pi)$ is a general polynomial manifold realising the infra-solvmanifold with fundamental group $\Pi$, then it still follows from the result of Der\'e that any \sf $f$ has a polynomial homotopy lift. 
So there exists a polynomial map $p:\R^h \to \R^h$ inducing a 
map $\bar{p}:M\to M$ which is homotopic to $f$ and hence $N(f)=N(\bar{p})$. However, it is no longer true that the terms 
$\det(I-J(\rho(\alpha) \circ p)_{x_0})$ are independent of $x_0$ and a fortiori we cannot express $N(f)$ with a formula like the one in the previous theorem. 
\end{remark}

\begin{example}
Let $\Pi=\Z^5 \rtimes \Z$ be the group from Example~\ref{big example} and consider the torsion free filtration
\[ 1 \nor \Pi_2\nor \Pi_1 \nor_f \Pi \mbox{ with }\Pi_2=\Z^5 \mbox{ and } \Pi_1= K = \Z^5 \rtimes 2 \Z .\]
(Note that $K$ is a fully invariant net subgroup of $\Pi$).

One can check that the map
\[ \rho : \Pi \to {\rm P}(\R^6): (\vec{z_2}, z_1) \mapsto \rho(\vec{z_2}, z_1) \mbox{ with }
\rho(\vec{z_2}, z_1) \begin{pmatrix}x_1 \\ \vec{x}_2\end{pmatrix} \mapsto 
\begin{pmatrix}
x_1 + z_1 \\
A^{z_1} \vec{x}_2 + z_2
\end{pmatrix}
\]
(where $x_1\in \R$ and $\vec{x}_2 \in \R^5$) is a canonical type polynomial (in fact even affine) representation of $\Pi$ with respect to the above torsion free filtration. 
So the solvmanifold with fundamental group $\Pi$ is diffeomorphic to the polynomial (affine) manifold $\R^6/\Pi$ where $\Pi$ is acting on $\R^6$ via $\rho$. 

\medskip

For any integer $k$ we consider the map 
\[ p_k: \R^6 \to \R^6 : \begin{pmatrix}x_1 \\ \vec{x}_2\end{pmatrix} \mapsto \begin{pmatrix} - x_1 \\ B_k \vec{x}_2\end{pmatrix} \]
where $B_k$ is the matrix from Example~\ref{big example}. Now, one can check that $p_k \circ \rho(\vec{z}_2, z_1) = 
\rho( B_k \vec{z}_2 , -z_1) \circ p_k = \rho( \varphi_k(\vec{z}_2, z_1)) \circ p_k$. This shows that $p_k$ induces a map $f_k$ on the quotient manifold $\R^6/\Pi$ (so $p_k$ is a lift of $f_k$) and that $f_k$ induces the endomorphism $\varphi_k$ on $\Pi$.  We already computed $N(f_k)$ using Corollary~\ref{formc} in Example~\ref{big example}.

Let us now redo the computations using Theorem~\ref{aveforusingp}. First of all note that 
\[ J(\rho(\vec{z}_2, z_1) \circ p_k)_{x_0} = \begin{pmatrix} 
-1 & 0 \\
0 & B_k
\end{pmatrix}\]
for any choice of $x_0\in \R^6$. Then according to Theorem~\ref{aveforusingp} we have that 
\begin{eqnarray*}
N(f_k) & = & \dfrac{1}{[\PI:K]}\sum_{\bar \al\in\PI/K}\lvert\det(I-J(\rho(\alpha)\circ p)_{x_0})|\\
           & = & \dfrac{1}{2} ( | \det \begin{pmatrix} 1 - (-1) & 0 \\
                                                                                  0 & I-B_k\end{pmatrix} |      +
                                         | \det \begin{pmatrix} 1 - (-1) & 0 \\
                                                                                  0 & I-AB_k\end{pmatrix} |  )    \\
           & = & \di{B_k} + \di{AB_k}         \\
           & = & 3 |1-k| + 3 |1 +k|                             
          \end{eqnarray*}
which is exactly the same result as we obtained before.
\end{example}

\section{fixed point properties of polynomial maps}\label{6}

In this section, we show that polynomial maps in some sense realise the least number of fixed points prescribed by the Nielsen number. 
We first focus on self-maps of \NR solvmanifolds.

\subsection{Fixed points of polynomial maps on \NR solvmanifolds.}

In this subsection, we show the following
\begin{prop}\label{realisenr}
Let $\hat p$ be a polynomial map on a canonical type polynomial \NR solvmanifold $\R^h/K$.
\begin{enumerate}
    \item If $N(\hat p)\neq 0$, then $\hat p$ has exactly $N(\hat p)$ fixed points.
    \item If $N(\hat p)=0$, then $\hat p$ has no fixed points or uncountably many fixed points.
\end{enumerate}
\end{prop}

We make some preliminary observations.

\begin{lemma}\label{rnsolv}
Let $f:S\to S$ be a self-map on an \NR solvmanifold $S$. Then
$$N(f)=\begin{cases}
R(f) &\text{if }R(f)<\infty,\\
0   &\text{if }R(f)=\infty.
\end{cases}$$
\end{lemma}

\begin{notation}
We can write the above more compactly: 
define the bijection
$\0 \cdot:\N_0\cup\{\infty\}\to\N$ 
by
$\0x=x$ if $x\in\N_0$ and $\0\infty=0$.
Set $\oo\cdot:=\0\cdot^{-1}$.
Then the above can be written as $N(f)=\0{R(f)}$ or $\oo{N(f)}={R(f)}$.
\end{notation}

\begin{proof}
Let $K$ be the fundamental group of $S$ and $\p\in\End K$ an endomorphism induced by $f$.
Let $\{\Delta_i,A_i\}$ be the linearisation corresponding to $K\nor_f \PI$, and let $\{F_0,\dots, F_c\}$ be  the linearisation of $f$. Recall from Theorem~\ref{prod} that $N(f)=\prod_{i=0}^c\di {F_i}$.

As $\p$ leaves $N$ invariant and $K/N$ is free abelian, we can compute $R(f)=R(\p)$ using the \emph{addition formula} \cite[Lemma~2.1]{dtv19}: 
say $\RR{\bar\p}=\{\cl{z_1}{F_0},\dots, \cl{z_d}{F_0}\}$ where $z_i\in K/N$ and $d=R(F_0)=\di{F_0}_\infty$. 
Take $k_i\in K$ with $k_iN=z_i$, and put $\tau_{k_i}:K\to K:x\mapsto k_ix\ii {k_i}$. Then
\begin{align*}
    R(\p)
    &=\sum_{i=1}^d R({\tau_{k_i}}|_{N}\circ\p|_N)\\
    &=\sum_{i=1}^d \prod_{j=1}^c R({A_j( k_i)}\circ F_j)\\
    &=\sum_{i=1}^d \prod_{j=1}^c \di{A_j( k_i)\circ F_j}_\infty\\
    &=\sum_{i=1}^d \prod_{j=1}^c \di{F_j}_\infty\\
    &=\oo{N(f)}.
\end{align*}
The second equality follows from repeated applications of \cite[Lemma~2.1]{dtv19}.
For the fourth equality, see Remark~\ref{nr}.
\end{proof}

\begin{lemma}\label{fpnr}
Let $\hat p:\S\to \S$ be a polynomial map on the canonical type polynomial \NR solvmanifold $\S$ with polynomial homotopy lift $p:\R^h\to\R^h.$ Let $k\in K$. 
\begin{enumerate}
    \item If $N(\hat p)\neq 0$, then $kp$ has a unique fixed point.
    \item If $N(\hat p)=0$, then $kp$ has no fixed points or uncountably many fixed points. 
\end{enumerate}
\end{lemma}
\begin{proof}
We use the notations of Section~\ref{pol}, and assume that $K\subseteq\pd{h}$ acts canonically (in the sense of Section~\ref{pol}) on $\R^h$.

By Theorem~\ref{thmcan} and Lemma~\ref{vormp}, we know that 
$$kp\vc{x_1}{x_{c+1}}
=\vcl{B_1(k)x_1}{\al_1}{B_{c+1}(k)x_{c+1}}{\al_{c+1}(x_1,\dots, x_{c})}$$ 
with $B_i(k):=A_{i-1}(k)F_{i-1}$ and $\al_i\in\pp{K_{i-1}}{k_i}$.

\begin{enumerate}
    \item 
    First, suppose that $N(\hat p)=\prod_{i=0}^c\di 
    {F_i}\neq 0$. As $K$ is $\mathcal{NR}$, also $\di{B_i(k)}=\di{F_{i-1}}\neq 0$ for all $k$. We conclude that each $kp$ has exactly one fixed point.
    
    \item Next, suppose that $N(\hat p)=0$. Then there exists $j\in\{1,\dots, c+1\}$ with $\di{B_j}=0$. This leaves two possibilities:
    
    \begin{itemize}
        \item $\fix{kp}$ is empty;
        \item $\fix{kp}$ is uncountable. Indeed, take $x\in\fix{kp}$ and let $m$ be the largest $i$ with $\di{B_i}=0$. 
        Take any $r\in \ker(I-B_m)$. As $\det(I-B_i)\neq 0$ for all $i>m$, there exist $y_{i}\in\R^{k_i}$, $i>m$, with
        $$y_{i}=B_i(y_{i})+p_i(x_1,\dots, x_{m-1},x_m+r, y_{m+1},\dots, y_{i-1}).$$ 
        So $kp$ fixes the point $(x_1,\dots, x_{m-1}, x_m+r, y_{m+1},\dots, y_{c+1})$ as well.
    \end{itemize}
    We conclude that each $\fix{kp}$ is either empty of uncountably infinite.
\end{enumerate}
This finishes the proof.
\end{proof}

\begin{proof}[Proof of Proposition~\ref{realisenr}]
Take a polynomial homotopy lift $p:\R^h\to\R^h$, and let $\p$ be the corresponding endomorphism on $K$, i.e.\ $p\circ k=\p(k)\circ p$ for all $k\in K$.
As $$\fix{\hat p}=\bigsqcup_{\clp{ k}\in\RR \p}\pi(\fix{k p}),$$
it remains to examine the projections $\pi_k:\fix{kp}\to\pi(\fix{kp})$.
\begin{enumerate}
    \item If $N(\hat p)\neq 0$, each $\fix{kp}$ is a singleton by Lemma~\ref{fpnr}(1). Hence $$\#\fix{\hat p}=R(\p)=N(\hat p)$$ by Proposition~\ref{rnsolv}.
    
    \item If $N(\hat p)=0$, each $\fix{kp}$ is either empty or uncountable by Lemma~\ref{fpnr}(2). 
    As $\pi_k$ has countable fibers, 
    each $\pi(\fix{kp})$ must be empty or uncountable as well. 
    We conclude that $\fix{\hat p}$ is empty or uncountably infinite.
\end{enumerate}
This completes the proof.
\end{proof}

\subsection{Fixed points of poynomial maps on \infra manifolds}

Proposition~\ref{realisenr} fails for polynomial maps on \infra manifolds, see Example~\ref{vb} below. 
We do have the following weaker version:
\begin{prop}\label{realiseinfra}
Let $f$ be a polynomial map on a canonical type polynomial \infra manifold $\R^h/\PI$. 
Then $f$ has either (uncountably) infinitely many fixed points, or $f$ has exactly $N(f)$ fixed points. 
\end{prop}
\begin{proof}
Let $p:\R^h\to\R^h$ be the polynomial map inducing $f$ on $\R^h/\PI$. 
Take a fully invariant, finite index \NR subgroup $K$ of $\PI$, and let $\hat p$ denote the induced map on $\R^h/K$. 
We thus have, for every $\alpha$ in $\PI$, the following commutative diagram:
$$ \begin{CD}
  \R^h @>\alpha\circ p>>\R^h \\
  @V \pi_K VV @VV \pi_K V\\
  \S @>\hat \alpha\circ \hat p>> \S\\
  @V \pi' VV @VV \pi' V\\
  \X @>\bar \alpha\circ f>> \X
\end{CD}$$
The fixed point set of $f$ decomposes as
$$\fix{f}=\bigcup_{\hat \alpha \in\PI/K} \pi'(\fix{\hat{\al}\hat p}).$$
Note that $\pi'$ has finite fibers. Hence, if $\fix{\hat\al\hat p}$ is uncountably infinite for some $\hat \alpha$ in $\PI/K$, so is $\fix{f}$. 

So suppose that $\fix{\hat\al\hat p}$ is finite for all $\hat\al \in\PI/K$. 
By Proposition~\ref{realisenr}, this implies in particular that $\#\fix{\hat\al\hat p}=0$ if $N(\hat\al\hat p)=0$. Then
\begin{align*}
    \#\fix{f}
    &=\#\bigcup_{\hat\al\in\PI/K} \pi'(\fix{\hat\al\hat p})\\
    &\leq\sum_{\hat\al\in\PI/K} \#\fix{\hat\al\hat p}\\
    &=\sum_{\hat\al\in\PI/K} N({\hat\al\hat p})\\
    &=N(f).
\end{align*}
As always $N(f)\leq \# \fix f$, we conclude that $f$ has either (uncountably) infinitely many fixed points, or precisely $N(f)$ fixed points.
\end{proof}

By \cite[Corollary~7.6]{fl13-1}, if $f$ is a \sf on an \infra manifold of type (R) and $R(f)<\infty$, then $N(f)=R(f)$. This result generalises to any \infra manifold. We first prove the following lemma:

\begin{lemma}\label{ongelijkheid}
Let $G$ be a group and let $\p$ be an automorphism of $G$. Suppose that $H$ is a normal, finite index subgroup of $G$ such that $\p(H)\subseteq H$. 
Then $$R(\p) \leq \sum_{x H \in G/H} R((\tau_x\circ \p)_{|_H})$$
with $\tau_x:G\to G:g\mapsto x g x^{-1}$ for all $x$ in $G$.

\end{lemma}

\begin{proof}
Note that if $x H=y H$, also $R((\tau_x\circ \p)_{|_H})=R((\tau_y\circ \p)_{|_H})$ by \cite[Cor.\ 3.2]{flt08}. 

Suppose that $\{x_1H,\dots,x_kH\}=G/H$ with $k=[G:H]$. 
Write $\p_i:= (\tau_{x_i}\circ \p)_{|_H}$ for every $i \in \{1,\dots, k\}$.  For every $X\in H/\sim_{\p_i}$, take an element $h_X\in H$ representing the $\p_i$-twisted conjugacy class $X$.
We will show that every element $g$ of $G$ is $\sim_\p$-equivalent with $h_Xx_i$ for some $i\in\{1,\dots, k\}$ and $X\in H/\sim_{\p_i}$.

Take $g \in G$. Say $g H=x_i H$ for some $i\in \{1,\dots, k\}$; write $g=hx_i$ with $h\in H$. Let $X:=[h]_{\p_i}\in H/\sim_{\p_i}$ be the Reidemeister class of $h$. There thus exists $z\in H$ such that 
 $h= z\, h_X \p_i(z^{-1})=z\, h_X \,x_i\p(z^{-1})x_i^{-1}$. Multiplying by $x_i$ on the right gives 
 $g=hx_i=z\, h_Xx_i\,\p(z^{-1})$, so that $g\sim_\p h_Xx_i$.
 
 We conclude that $\p$ has at most $R(\p_1)+\dots + R(\p_k)$ Reidemeister classes.
\end{proof}

\begin{prop}\label{rninfrasolv}
Let $f:M\to M$ be a self-map on an \infra manifold $M$. 
If $R(f)<\infty$, then $N(f)=R(f)$.
\end{prop}

\begin{proof}
Write
$M=\R^h/\PI$, with $\PI=\pi(M)$; then $f$ induces an endomorphism $\p$ on $\PI$. 
Take $K\nor \PI$ a finite index, fully invariant \NR subgroup of $\PI$. 
For any $\alpha\in \PI$, the map $\alpha\circ f$ lifts to a map $\hat\alpha\circ\hat f$ on $\R^h/K$ inducing $\tau_\al|_K\circ \p|_K$ on $K$. 

By the averaging formula, Theorem~\ref{avefor},
we know that 
\begin{align*}
N(f)&=\dfrac{1}{[\PI:K]}\sum_{\bar\alpha\in\PI/K} N(\hat\alpha\circ\hat f)\\
&=\dfrac{1}{[\PI:K]}\sum_{\bar\alpha\in\PI/K} \0{R(\hat\alpha\circ\hat f)}\\
&=\dfrac{1}{[\PI:K]}\sum_{\bar\alpha\in\PI/K} \0{R((\tau_\alpha\circ\p)|_K)}.
\end{align*}

As $R(\p)<\infty$ and $\PI/K$ is finite, also $R((\tau_\alpha\circ\p)|_K)<\infty$ for all $\alpha \in\PI$. 
Indeed, if $R((\tau_\alpha\circ\p)|_K)$ were infinite, $R(\tau_\al\circ \p)=R(\p)$ would be infinite as well \cite[Lemma~1.1]{gw09-2}.
Hence
$$N(f)=\dfrac{1}{[\PI:K]}\sum_{\bar\al\in\PI/K}R((\tau_\al\circ\p)|_K)
.$$
The proposition thus follows from \cite[Theorem~3.5]{hlp12-1} once we show that $\F{\tau_\alpha\circ\p}\subseteq K$ for all $\alpha\in\PI$.
We will show in fact that $\mathrm{fix}(\tau_\alpha\circ\p)$ is trivial. 
  As $K\f \PI$ and $\PI$ is torsion free, it  suffices to show that $\mathrm{fix}((\tau_\alpha\circ\p)|_K)$ is trivial.

In keeping with our previous notation, let $\{A_0(\alpha)\circ F_0, \ldots ,A_c(\alpha)\circ F_c\}$ denote the linearisation of $(\tau_\alpha\circ\p)|_K$. From 
$$\infty> R((\tau_\alpha\circ\p)|_K)=\prod_{i=0}^c \,\di{A_i(\alpha)\circ F_i}_\infty,$$ it easily follows that $\det(I-A_i(\alpha)\circ F_i)\neq 0$ for all $i\in\{0,\ldots c\}.$ Hence $\mathrm{fix}(A_i(\alpha)\circ F_i)=1$ for all $i$, implying $\mathrm{fix}((\tau_\alpha\circ\p)|_K)=1$ as well.  
\end{proof}

We conclude this section with the following example.

\begin{example}\label{vb}
Consider the group $\PI:=\Z\rtimes_{-1}\Z.$ Then $K:=\Z\times 2\Z\cong \Z^2$ is a fully invariant \NR subgroup of index $2.$
We can realise $\PI$ as a subgroup of $\pd{2}$ via
$$\rho:\PI\to\pd{2}:(k,n)\mapsto \quad 
\begin{pmatrix}
r_1\\r_2
\end{pmatrix}
\mapsto
\begin{pmatrix}
(-1)^nr_1+k\\r_2+n
\end{pmatrix}.$$
Note that $\rho:\PI\to\pd 2$ is of canonical type with respect to $1\nor K\nor_f \PI$.

For ease of notation, we write $(k,n)$ also as $z^kt^n$. 
For every $a\in\Z$ and $c\in 2\Z+1$, we can define an endomorphism $\p$ by $\p(z)=z^a$ and $\p(t)=t^c$. 
Then
\begin{itemize}
    \item $\p':=\p|_K$ sends $z\mapsto z^a$ and $t^2\mapsto t^{2c} $;
    
    \item $\tau_t:K\to K$ sends $z\mapsto \ii z$ and $t^2\mapsto t^2$;
    
    \item $\tau_t\circ \p'$ sends $z\mapsto z^{-a}$ and $t^2\mapsto t^{2c}$.
\end{itemize}
It is easily verified that
$$p:\R^2\to\R^2: \begin{pmatrix}
r_1\\r_2
\end{pmatrix}
\mapsto
\begin{pmatrix}
ar_1\\cr_2
\end{pmatrix}$$
satisfies $p\circ\rho(\gamma)=\rho(\p(\gamma))\circ p$ for every $\gamma\in\PI$. We thus have maps 

$$ \begin{CD}
  \R^h @>p>>\R^h \\
  @V \pi_K VV @VV \pi_K V\\
  \S @>\hat p>> \S\\
  @V \pi' VV @VV \pi' V\\
  \X @>f>> \X
\end{CD}$$
inducing $\p$ and $\p'$.
Now 
\begin{align*}
    N(f)
    &=\sum_{\hat\alpha \in \PI/K}\di{\tau_\al\circ\p'}\\
    &=|(1-a)(1-c)|+|(1+a)(1-c)|
\end{align*}
\begin{itemize}
    \item Take $a=c=1$. Then $N(f)=0$. As $\fix{p}=\fix{\mathrm{Id}}$ is infinite, so is $\fix{f}$. 
    
    \item Take $a=c=-1$. Then $N(f)=4$, but $\fix{f}$ is infinite as well, as $\fix{\rho(t)\circ p}$ is infinite:
    $$\rho(t)\circ p
    \begin{pmatrix}
    r_1\\r_2
    \end{pmatrix}
    =
    \rho(t)
    \begin{pmatrix}
    -r_1\\-r_2
    \end{pmatrix}
    =
    \begin{pmatrix}
    r_1\\-r_2+1
    \end{pmatrix}.$$
    
    Note that, although $N(f)\neq 0$, the map $f$ has $R(f)=\infty.$ Indeed, the subgroup
    $G:=\langle z \rangle$ is invariant under $\p$ and $\PI/G\cong\Z$, so we can use \cite[Lemma~2.1]{dtv19} to compute that
    
    \begin{align*}
        R(\p)
        &=\sum_{[\hat\al]\in\RR{-\mathrm{Id}|_\Z}} R(\tau_\alpha\circ\p|_{\Z})\\
        &=R(-\mathrm{Id}|_{\Z})+R(-\mathrm{Id}|_{\Z}\circ -\mathrm{Id}|_{\Z})\\
        &=R(-\mathrm{Id}|_{\Z})+R(\mathrm{Id}|_{\Z})\\
        &=2+\infty\\
        &=\infty.
    \end{align*}
\end{itemize}
\end{example}

\appendix
\section*{Appendix}
We prove:
\begin{lomma}
Let $X\in\Z^{n\times n}$ and $\Phi\in \Q^{m\times m}$ be matrices and let $A:\Z^m\to\operatorname{SL}_n(\Z)$ be an endomorphism such that $A(v)$ is net for all $v\in\Z^m$.
Suppose that $\Phi$ does not have $1$ as an eigenvalue, and that there exists $k\in\N$ such that $\Phi(k\Z^m)\subseteq \Z^m$ and $XA(kv)=A(\Phi(kv))X$ for all $v\in\Z^m$. 
Then $\de{A(v)X}=\de{X}$ for all $v\in\Z^m$. 
\end{lomma}
\begin{proof}
We follow the matrix analysis carried out by Keppelmann and McCord 
\cite[Section~4]{km95-1}.
This analysis consist of three steps.

\stepb{Step 1.} Reduction to the unipotent and semisimple case.

\medskip

For $v\in\Z^m$, write $A(v)=U(v)T(v)$ with $U(v)$ unipotent, $T(v)$ semisimple and $[U(v),T(v)]=1$. 
This defines morphisms $U$, $T:\Z^m\to\mathrm{GL}_n(\Q)$ and $[U(v),T(w)]=1$ for $v\neq w\in \Z^m$, too.
We show that it is sufficient to prove the lemma for $A=U$ ({the unipotent case}) and $A=T$ ({the semisimple case}).

Following \cite{km95-1}, we say that \textit{$X$ \ap commutes with $A$} if there exists $k\in\N$ such that $XA(kv)=A(\Phi(kv))X$ for all $v\in\Z^m$.

\stepi{Observation: $X$ \ap commutes with both $U$ and $T$.}

\medskip

\medskip

Set $B(v):=\left(\begin{smallmatrix} A(kv)&0\\  0&A (\Phi(kv))\end{smallmatrix}\right)$ for all $v$ in $\Z^m$.
Then 
$$B_u(v):=\left(\begin{smallmatrix} U(kv)&0\\  0&U (\Phi(kv))\end{smallmatrix}\right)\quad  
\text{and} 
\quad B_s(v):=\left(\begin{smallmatrix} T(kv)&0\\  0&T (\Phi(kv))\end{smallmatrix}\right)$$ 
are the unipotent and semisimple part of $B(v)$, respectively.
Consider the subgroup $B:=\{B(v)\mid v\in\Z^m\}$ of $\GL{2n}$.
The relation $XA(kv)=A(\Phi(kv))X$ is polynomial in the coefficients of $B(v)$. 
As $B_u(v)$ and $B_s(v)$ are contained in the Zariski closure of $B$, they too must satisfy this relation. So $X$ \ap commutes with $U$ and with $T$.
\eindebewijs

\medskip

Suppose now that 
$$\begin{aligned}
    &(\ast)\ \de{U(v)M}=\de{M}\text{ for every } M \text{ \ap commuting with } U;\\
    &(\star)\  \de{T(v)M}=\de{M} \text{ for every }M \text{ \ap commuting with }T.
\end{aligned}$$

Say $k$ satisfies $XU(kv)=U(\Phi(kv))X$ for all $v\in\Z^n$. Then for all $v, w\in\Z^n$,
\begin{align*}
    T(v)XU(kw)&=T(v)U(\Phi(kw))X\\
    &=U(\Phi(kw))T(v)X,
\end{align*}
so in particular, $T(v)X$ \ap commutes with $U$.
Hence
\begin{align*}
    \de{A(v)X}
    &=\de{U(v)T(v)X}\\
    &\overset{(\ast)}{=}\de{T(v)X}\\
    &\overset{(\ast\ast)}{=}\de{X}.
\end{align*}
 We prove ($\ast$) in step 2 and ($\ast\ast$) in step 3.

\stepb{Step 2.} The unipotent case.

\medskip

\medskip

Suppose that $M\in M_n(\C)$ \ap commutes with $U$, say $d\in\N$ satisfies $MU(dv)=U(\Phi(dv))M$ for all $v\in\Z^m$. From the original version of this lemma, \cite[Theorem~4.2]{km95-1}, we already know that
\begin{align*}
\de{U(d\,v)M}&=\de{M}\\
\intertext{for all $v\in\Z^m$. Fix $x\in\Z^m$. Then for all $z\in\Z$, also}
\de{U(dz\, x)M}&=\de{M}. 
\end{align*}
However, as $U(x)$ is unipotent, the entries of $U(t\, x)=U(x)^{t}$, $t\in\Z$, are polynomials in $t$ (depending on the entries of $U(x)$, of course). 
Hence $$\de{U(t\, x)M}-\de{M}\in\Q[t]$$ is a polynomial vanishing on $d\Z$, so it must be zero. We conclude that $\de{U(x)M}=\de{M}$, as required.

\stepb{Step 3.} The semisimple case.

\medskip

Suppose that $M\in M_n(\C)$ \ap commutes with $T$; we are to show that $\de{T(v)M}=\de{M}$ for all $v\in\Z^m$.

Choose a basis of common eigenvectors $\{f_j\}_{j=1,\ldots, n}$ of the $T(v)$'s. 
Let $[x_{ij}]\in M_n(\C)$ represent $M$ with respect to the basis $\{f_j\}$, and let $\lambda_j(v)$ be the eigenvalue of $T(v)$ associated to $f_j$. 
Then, in this notation, 
$$\de{T(v)M}=\sum_{\sigma\in\Sn}\left(\sgn \sigma\prod_{i=1}^n \, \delta_{i\sigma(i)}-\lambda_i(v)\, x_{i\sigma(i)}\right)$$
denoting $\delta_{ij}$ the Kronecker delta.
It is therefore sufficient to prove the following:

\stepi{Reduction 1. } $\forall \sigma\in\Sn: \prod_{i=1}^n \, \delta_{i\sigma(i)}-\lambda_i(v) \, x_{i\sigma(i)}=\prod_{i=1}^n \, \delta_{i\sigma(i)}-x_{i\sigma(i)}$

\medskip

Take $\sigma\in\Sn.$ Write $\sigma=\sigma_1\circ \cdots \circ \sigma_l$ as a product of disjoint cycles $\sigma_i$.
For each element in $\fix \sigma$, we add a `cycle' of length $1$. Formally, denoting $\fix\sigma=\{e_1,\ldots, e_d\}$ with $d=\#\fix\sigma$, we set $\sigma_{l+j}:=(e_j)$ for $j\in\{1,\ldots, d\}$. 

We can now partition the set $\{1,\ldots, n\}$ according to these cycles: define
$$V(\sigma_j):=
\begin{cases}
\{1,\ldots,n\}\setminus\fix{\sigma_j} 
&\text{ if $j\leq l$,}\\
\{e_j\} &\text{ if $j>l$.}
\end{cases}$$
Setting $r=d+l$, we see that 
$\{1,\ldots, n\}=\cup_{j=1}^r V(\sigma_j)$,
so we can further reduce to 

\stepi{Reduction 2. } $\forall j\in\{1,\ldots r\}: \displaystyle\prod_{i\in V(\sigma_j)}  \delta_{i\sigma_j(i)}-\lambda_i(v) \, x_{i\sigma_j(i)}=\displaystyle\prod_{i\in V(\sigma_j)} \delta_{i\sigma_j(i)}-x_{i\sigma_j(i)}$

Take $j\in\{1,\ldots, r\}$. 
Write
$$\sigma_j=\left(h\  \sigma_j(h)\  \sigma_j^2(h)\  \ldots \  \sigma_j^{s-1}(h)\right).$$
To shorten notation, we write $\sigma^i:=\sigma_j^i(h)$. 
So $\sigma^{s+1}=\sigma^1$.
We have to show that
\begin{equation}
\prod_{i=1}^s \delta_{\sigma^i\sigma^{i+1}}-\lambda_{\sigma_i}(v)\, x_{\sigma^i\sigma^{i+1}}
=
\prod_{i=1}^s \delta_{\sigma^i\sigma^{i+1}}-x_{\sigma^i\sigma^{i+1}}. \tag{$\circ$}    
\end{equation}

We examine this statement more closely by distinguishing the cases $s=1$ and $s>1$.
\begin{itemize}
\item If $s=1$, statement ($\circ$)
reads $1-\lambda_h(v)\,x_{hh}=1-x_{hh}$, or equivalently, $x_{hh}=0$ or $\lambda_h(v)=1$.

\item If $s>1$, statement ($\circ$)
reads $\prod_{i=1}^s\lambda_{\sigma^i}(v)\,x_{\sigma^i\sigma^{i+1}}=\prod_{i=1}^s x_{\sigma^i\sigma^{i+1}}$, or equivalently, $\prod_{i=1}^s x_{\sigma^i\sigma^{i+1}}=0$ or $\prod_{i=1}^s\lambda_{\sigma^i}(v)=1$.

\end{itemize}
So in both cases, statement ($\circ$) is equivalent to $\prod_{i=1}^s x_{\sigma^i\sigma^{i+1}}=0$ or $\prod_{i=1}^s\lambda_{\sigma^i}(v)=1$. We will show the following:

\stepi{Reduction 3. If $\prod_{i=1}^s x_{\sigma^i\sigma^{i+1}}\neq 0$, then $\prod_{i=1}^s\lambda_{\sigma^i}(v)=1$.}

\medskip

So, assume that $\prod_{i=1}^s x_{\sigma^i\sigma^{i+1}}\neq 0$. As $M$ \ap commutes with $T$, there exists $k\in\N$ such that $MT(kw)=T(\Phi(kw))M$ for every $w\in\Z^m$. Take $i\in\{1,\ldots, s\}$. 
Then
\begin{align*}
    MT(kw)f_{\sigma^{i+1}}
    &=M\lambda_{\sigma^{i+1}}(kw)f_{\sigma^{i+1}}\\
    &=\sum_{j=1}^n \lambda_{\sigma^{i+1}}(kw)x_{j\sigma^{i+1}}f_j\\
\intertext{and}
    T(\Phi(kw))Mf_{\sigma^{i+1}}
    &=T(\Phi(kw))\left(\sum_{j=1}^nx_{j\sigma^{i+1}}f_j\right)\\
    &=\sum_{j=1}^n x_{j\sigma^{i+1}} T(\Phi(kw))(f_j)\\
    &=\sum_{j=1}^n x_{j\sigma^{i+1}} \lambda_j(\Phi(kw))f_j.
\end{align*}
Equating $MT(kw)=T(\Phi(kw))M$, we see that $ x_{j\sigma^{i+1}}\neq 0$ implies  $\lambda_{\sigma^{i+1}}(kw)=\lambda_j(\Phi(kw))$ for all $j\in\{1,\ldots, n\}$. 
As $x_{\sigma^i\sigma^{i+1}}\neq 0$ by assumption, 
\begin{equation}
    \lambda_{\sigma^{i+1}}(kw)=\lambda_{\sigma^i}(\Phi(kw))
    \tag{$\bullet$}
\end{equation}
for all $i\in \{1,\ldots , s\}$ and $w\in\Z^m$. 

Take a basis $\{e_1,\ldots, e_m\}$ of $\Z^m$ such that $\{e_1,\ldots, e_q\}$ (viewed as subset of $\R^m)$ spans $\ker((\Phi^T)^s-I)$ for some $q\in\{0,\ldots, m\}$.
Write $\Phi=[\phi_{ij}]$ with respect to the basis $\{e_j\}$, and write $\lambda_{i,j}:=\lambda_{\sigma^i}(e_j)$. 
(So again $\lambda_{s+1,j}=\lambda_{1,j}$.)

We have to show that $\prod_{i=1}^s\lambda_{\sigma^i}(v)=1$ for every $v\in\Z^m$. 
If $v=\sum_{j=1}^m \alpha_j e_j$, 
then $T(v)=\prod_{j=1}^m T(e_j)^{\alpha_j}$, 
hence $\lambda_{\sigma^i}(v)=\prod_{j=1}^m \lambda_{i,j}^{\alpha_j}$ for every $i\in\{1,\ldots, s\}$.
Therefore, it is sufficient to prove that

\stepi{Reduction 4. $\prod_{i=1}^s \lambda_{i,j}=1$ for all $j\in\{1,\ldots ,m\}$.}

\medskip

We know exploit condition ($\bullet$). Thereto, take $j\in\{1,\ldots, m\}$ and choose $r_{i,j}$, $\theta_{i,j}\in\R$, $i\in\{1,\ldots, s\}$, satisfying $\lambda_{i,j}=e^{r_{i,j}+2\pi i \,\theta_{i,j}}$. 
In this notation,
\begin{align*}
\lambda_{\sigma^{i+1}}(ke_j)
&={\lambda_{i+1, j}}^k\\
&=e^{kr_{i+1,j}}e^{2\pi i\, k\theta_{i+1,j}}\\
\intertext{when we agree that $r_{s+1,j}:=r_{1,j}$ and $\theta_{s+1,j}:=\theta_{1,j}$. Furthermore, since $\Phi(ke_j)=\sum_{\alpha=1}^m k\phi_{\alpha j} e_\alpha$,}
\lambda_{\sigma^i}(\Phi(ke_j))
&=\prod_{\alpha =1}^m\lambda_{i,\alpha}^{k\phi_{\alpha j}}\\
&=\prod_{\alpha=1}^m e^{k\phi_{\alpha j}r_{i,\alpha}+2\pi i \, k\phi_{\alpha j}\theta_{i, \alpha}}\\
&=e^{k\sum_{\alpha=1}^m \phi_{\alpha j }r_{i,\alpha}}
e^{2\pi i\, k \sum_{\alpha=1}^m\phi_{\alpha j}\theta_{i,\alpha}}.
\end{align*}
Imposing condition ($\bullet$) implies that for all $i\in\{1,\ldots, s\}$, 

\begin{align*}
kr_{i+1, j}&=k\sum_{\alpha=1}^m \phi_{\alpha j }r_{i,\alpha}
& 
&\text{and}
&
k\theta_{i+1, j} &\equiv k\sum_{\alpha=1}^m \phi_{\alpha j }\theta_{i,\alpha} \mod \Z.\\
\intertext{Define $R_i$, $\Theta_i\in\R^m$ as the vectors with $j$-th component equal to $r_{i,j}$ and $\theta_{i,j}$, respectively. Then}
R_{i+1}&=\Phi^T R_i
&
&\text{and}
&
\Theta_{i+1}&\equiv\Phi^T\Theta_i \mod \Q^m.
\end{align*}
Note that $R_{s+1}=R_1$ and $\Theta_{s+1}=\Theta_1$. Therefore, the above implies that
\begin{itemize}
\item $\Phi^T\left(\sum_{i=1}^s R_i\right)=\sum_{i=1}^s R_i$;
\item ${\Phi^T}^s(R_i)=R_i$;
\item ${\Phi^T}^s(\Theta_i)\equiv\Theta_i \mod \Q^m$.
\end{itemize}

It follows from the first item that $\sum_{i=1}^s R_i=0$ for $\Phi^{T}$ (and $\Phi$) does not have eigenvalue 1.
    
The second item implies that $R_i\in\ker({\Phi^T}^s-I)$, hence $r_{i,j}=0$ if $j>q$. As $\lambda_{i,j}$ cannot be a nontrivial root of unity, $\theta_{i,j}$ must either be $0$ or irrational if $j>q$.
In fact, $\theta_{i,j}=0$ as $({\Phi^T}^s-I)(\Theta_i)$ must lie inside $\Q^m$ and 
$${\Phi^T}^s-I=
\begin{pmatrix}
0 & *\\
0 & \Phi'
\end{pmatrix}$$
for some $\Phi'\in\mathrm{GL}_{m-q}(\Q)$ (with respect to the basis $\{e_j\}$). So ${\Phi^T}^s(\Theta_i)=\Theta_i$ as well.

However, then $\Phi^T$ fixes $\sum_{i=0}^{s-1}{\Phi^T}^i(\Theta_1)$, implying $\sum_{i=0}^{s-1}{\Phi^T}^i(\Theta_1)=0$ for $\Phi$ does not have $1$ as an eigenvalue. As $\Theta_{i+1}\equiv {\Phi^T}^i(\Theta_1) \mod \Q^m$, we conclude that $\sum_{i=1}^s \Theta_i\in \Q^m$. 

Translating $\sum_{i=1}^s R_i=0$ and $\sum_{i=1}^s \Theta_i\in \Q^m$ back to the $r_{i,j}$/$\theta_{i,j}$-notation gives 
$$\sum_{i=1}^s r_{i,j}=0 \quad\text{and}\quad \sum_{i=1}^s \theta_{i,j}\in\Q$$
for all $j\in\{1,\ldots, m\}$.
Hence
\begin{align*}
    \prod_{i=1}^s \lambda_{i,j}
    &=\prod_{i=1}^s e^{r_{i,j}+2\pi i\, \theta_{i,j}}\\
    &=e^{\sum_{i=1}^s r_{i,j} + 2\pi i\sum_{i=1}^s \theta_{i,j}}\\
    &\in e^{2\pi i \,\Q}
\end{align*}
is a root of unity. As $T(e_j)$ is net, this implies that $\prod_{i=1}^s \lambda_{i,j}=1$, concluding the proof.
\end{proof}

\label{References}

\end{document}